\documentclass[10pt,reqno]{amsart}
\usepackage{amssymb,mathrsfs,graphicx}
\usepackage{ifthen}
\usepackage[margin=1.0in]{geometry}
\usepackage{caption}
\usepackage{sidecap, dsfont}
\usepackage{rotating}

\usepackage{colortbl}
\definecolor{black}{rgb}{0.0, 0.0, 0.0}
\definecolor{red}{rgb}{1.0, 0.5, 0.5}
\provideboolean{shownotes} 
\setboolean{shownotes}{true} 
\newcommand{\margnote}[1]{
\ifthenelse{\boolean{shownotes}}%
{\marginpar{\raggedright\tiny\texttt{#1}}}%
{}%
}
\newcommand{\hole}[1]{
\ifthenelse{\boolean{shownotes}}%
{\begin{center} \fbox{ \rule {.25cm}{0cm} \rule[-.1cm]{0cm}{.4cm}
\parbox{.85\textwidth}{\begin{center} \texttt{#1}\end{center}} \rule
{.25cm}{0cm}}\end{center}} {} }

\graphicspath{{pics/}}

\title[Coupled Vlasov and non-Newtonian fluid dynamics]{Coupled Vlasov and non-Newtonian fluid dynamics: existence and large-time behavior}

\numberwithin{equation}{section}

\newtheorem{theorem}{Theorem}[section]
\newtheorem{lemma}{Lemma}[section]

\newtheorem{proposition}{Proposition}[section]
\newtheorem{remark}{Remark}[section]
\newtheorem{definition}{Definition}[section]

\newcommand{\R}{\mathbb R}
\newcommand{\N}{\mathbb N}

\newcommand{\ls}{\lesssim}

\newcommand{\T}{\mathbb T}

\newcommand{\mc}{\mathcal C}

\newcommand{\bq}{\begin{equation}}
\newcommand{\eq}{\end{equation}}
\newcommand{\e}{\varepsilon}
\newcommand{\te}{\tilde\varepsilon}
\newcommand{\lt}{\left}
\newcommand{\rt}{\right}

\newcommand{\pa}{\partial}

\newcommand{\calE}{\mathcal{E}}

\newcommand{\calD}{\mathcal{D}}
\newcommand{\mt}{\mathcal{T}}

\newcommand{\intt}{\int_{\T^3}}
\newcommand{\intr}{\int_{\R^3}}

\newcommand{\inttr}{\iint_{\T^3 \times \R^3}}

\begin{document}
\allowdisplaybreaks

\author[Choi]{Young-Pil Choi}
\address[Young-Pil Choi]{\newline Department of Mathematics \newline
	Yonsei University, 50 Yonsei-Ro, Seodaemun-Gu, Seoul 03722, Republic of Korea}
\email{ypchoi@yonsei.ac.kr}

\author[Jung]{Jinwook Jung}
\address[Jinwook Jung]{\newline 
	Department of Mathematics and Research Institute for Natural Sciences\newline
	Hanyang University, 222 Wangsimni-ro, Seongdong-gu, Seoul 04763, Republic of Korea}
\email{jinwookjung@hanyang.ac.kr}

\author[Wr\'oblewska-Kami\'nska]{Aneta Wr\'oblewska-Kami\'nska}
\address[Aneta Wr\'oblewska-Kami\'nska]{\newline Institute of Mathematics of the Polish Academy of Sciences \newline
	ul. \'Sniadeckich 8, 00-656 Warszawa, Poland}
\email{awrob@impan.pl}

\date{\today}

\subjclass[]{35Q70, 76A05}
\keywords{Kinetic-fluid model, non-Newtonian flow, global weak solutions, large-time behavior.}

\begin{abstract} We study a coupled kinetic-non-Newtonian fluid system on the periodic domain $\T^3$, where particles evolve by a Vlasov equation and interact with an incompressible power-law fluid through a drag force. We prove the global existence of weak solutions for all $p > \frac85$, where $p > 1$ denotes the power-law exponent of the fluid's stress-strain relation. Under an additional uniform boundedness assumption on the particle density, we also establish large-time decay of a modulated energy functional measuring deviation from velocity alignment. The decay rate is algebraic when $p > 2$ and exponential when $\frac65 \le  p \le 2$, reflecting the role of fluid dissipation in the large-time dynamics.
\end{abstract}

\maketitle \centerline{\date}

\tableofcontents

%
%
%
%

\section{Introduction}
The interplay between dilute particle systems and incompressible non-Newtonian fluids arises in various applications, such as polymeric solutions and suspension flows. These systems involve multiscale interactions between a kinetic population of particles and a macroscopic fluid medium. A mathematical model for such systems couples a Vlasov-type kinetic equation for the particle distribution with a non-Newtonian fluid described by a modified incompressible Navier--Stokes system.

In this paper, we consider the following coupled kinetic-fluid system:
\begin{align}\label{main_eq}
\begin{aligned}
&\pa_t f + v \cdot \nabla_x f + \nabla_v \cdot ((u-v)f) = 0,  \quad (x,v) \in \T^3 \times \R^3, \quad t > 0,\cr
&\pa_t u + (u \cdot \nabla_x)u + \nabla_x \pi - \mbox{div}_x (\tau(Du)) = -\int_{\R^3} (u-v)f\,dv, \cr
&\mbox{div}_x u = 0,
\end{aligned}
\end{align}
subject to initial data
\begin{align}\label{init_data}
f(0,x,v) = f_0(x,v), \quad u(0,x) = u_0(x), \quad x \in \T^3, \ v \in \R^3.
\end{align}
Here $f(t,x,v)$ is the particle distribution function, $u(t,x)$ is the fluid velocity, and $\pi(t,x)$ is the pressure.  The stress tensor $\tau$ depends on the symmetric part of the gradient, $Du := (\nabla_x u + \nabla_x^T u)/2$, and is assumed to satisfy the power-law relation
\[
\tau(Du) = \mu |Du|^{p-2}Du, \qquad \mu > 0,\quad p > 1,
\]
which describes generalized Newtonian fluids such as shear-thinning ($p<2$) and shear-thickening ($p>2$) models.

The mathematical study of kinetic-fluid systems has drawn considerable attention due to their relevance in modeling fluid-particle interactions in complex media. A prototypical example is the Navier--Stokes--Vlasov--Fokker--Planck system, which describes dilute particles dispersed in a viscous fluid, with applications ranging from spray dynamics and combustion to atmospheric processes such as rain formation and pollution transport \cite{BWC00, CGL08, Mat06, Oro81, RM52, RM52_2}.
 The behavior of a microscopic-macroscopic system for a general class of dilute solutions of polymeric liquids with noninteracting polymer chains can be described by Navier--Stokes--Fokker--Planck type equations with an elastic extra-stress tensor on the right-hand side in the momentum equation. The extra-stress tensor stems from the random movement of the polymer chains and is defined through the associated probability density function which satisfies a Fokker--Planck-type parabolic equation. The existence of global-in-time weak solutions for such problems can be found in \cite{BS2011,BS2012}.

In the Newtonian regime, where the fluid stress tensor is linear in the velocity gradient, the global-in-time existence of weak solutions to the incompressible Vlasov--Navier--Stokes or Navier--Stokes--Vlasov--Fokker--Planck systems has been established using compactness methods and velocity averaging techniques \cite{BDGM09, BGM17, CCK16, CKL11,  CJ23,LLY22, Yu13}; see also \cite{Ham98} for the Vlasov--Stokes case. For the incompressible Vlasov--Fokker--Planck--Euler system, classical solutions near equilibrium have been constructed with algebraic decay rates \cite{CDM11} in the whole space, while the Navier--Stokes--Vlasov--Fokker--Planck system admits global classical solutions under smallness assumptions, via hypocoercivity estimates \cite{GMZ10}.

The non-Newtonian case introduces additional analytical challenges due to the nonlinear dependence of the stress tensor on the strain rate. For power-law fluids with exponents $p \geq \frac{11}{5}$, global weak and strong solutions have been obtained via Galerkin approximation, monotonicity and compactness arguments.  For the theory of weak and strong solutions for power-law fluids without coupling with kinetic equations, see \cite{MNRR96}. Weak existence theories allowing behaviors more general than power-laws, such as anisotropic fluids with flows obeying nonstandard rheology, are also available in the framework of Musielak-Orlicz spaces, see \cite{CGSW2021, W2013}. 
For the case of non-Newtonian fluid coupled with particle systems, see \cite{MPP18, ZFG24}. These approaches typically involve approximations of the nonlinear stress and a priori estimates. Local well-posedness of strong solutions for the Vlasov--non-Newtonian system has also been studied under structural conditions on the stress tensor, using Wasserstein stability and fixed-point methods \cite{KKK23}. In parallel, global weak solutions were constructed using truncated approximations and energy-based arguments \cite{ZFG24}.

The large-time behavior of coupled kinetic-fluid systems has been studied in various settings. In the Newtonian case, algebraic and exponential decay estimates have been obtained for the Vlasov--Navier--Stokes system; see \cite{Choi16, Han22, HMM20}. For non-Newtonian couplings, an exponential rate of convergence has been established in alignment-type models using Lyapunov functional methods \cite{HKKP18, MPP18}. However, these results fundamentally depend on the presence of a linear viscous term in the fluid stress tensor. While a $p$-power nonlinearity is present in those models, it does not contribute critically to the decay mechanism. In contrast, our system features a purely nonlinear $p$-power law viscosity, without any linear dissipation. In this regime, exponential decay is not generally expected; instead, the large-time behavior is governed by the interplay between nonlinear fluid diffusion and kinetic drag, leading naturally to algebraic decay.
 
Despite these advances, most existing works either focus on strong solutions under restrictive assumptions or establish weak solutions without quantitative information on their large-time behavior. Moreover, previous results typically include a linear (Newtonian) component in the stress tensor to guarantee coercivity in $L^2$ and facilitate compactness arguments, while our model features a purely nonlinear $p$-power law viscosity without any linear part. Although such a nonlinearity still provides coercivity in the $L^p$ framework, its fully nonlinear nature poses new analytical challenges. To the best of our knowledge, this is the first global-in-time existence theory for weak solutions in this setting. Furthermore, existing large-time behavior analyses fundamentally rely on the linear viscous term to derive exponential decay rates, leaving the quantitative study of large-time dynamics in the purely nonlinear regime largely unexplored. Many of these results are also confined to bounded spatial domains with prescribed boundary conditions or the whole space, while the periodic case remains less explored.

The aim of this paper is twofold. First, we establish the global-in-time existence of weak solutions to the system \eqref{main_eq}--\eqref{init_data} in the periodic domain $\T^3$, under suitable assumptions on the initial data and for the extended exponent range $p > \frac85$, which notably includes the lower part of the shear-thinning regime. Second, we investigate the large-time dynamics of such systems and obtain quantitative decay estimates for the modulated energy in the case $p < 2$, which reflects the combined dissipative effects of the non-Newtonian viscosity and the coupling through the drag force.

Before we proceed further, we emphasize the analytic difficulties in establishing the global existence of weak solutions. One particular difficulty stems from the lack of direct control over the spatial average of the fluid velocity. In the case of pure non-Newtonian fluids, Poincar\'e-type inequalities can be employed either via homogeneous Dirichlet boundary conditions in bounded domains or by imposing a zero-average constraint in the periodic setting. However, when coupled to a kinetic equation, the total momentum of the system is conserved, and the spatial averages of the fluid and particle velocities may evolve in time. As a result, standard coercivity estimates involving $\|\nabla u\|_{L^p}$ cannot be directly applied. Addressing this issue requires a refined analysis of the coupled dynamics and careful treatment of the interaction between fluid dissipation and particle momentum exchange.

Another challenge is to control the drag force under the limited compactness available for the fluid velocity. Due to the purely $p$-power viscosity, the energy estimates provide only $W^{1,p}$-regularity to the fluid velocity. Thus, when $p<2$, the strong compactness of the velocity field is weaker than in the Newtonian case and is available only in the space $L^q(\T^3)$ with $q < \frac{3p}{3-p}$. Although this limited compactness is already present in pure non-Newtonian fluids, the coupled system requires an additional compactness analysis to pass to the limit in the drag force term. More precisely, one has to verify that the weakened compactness of the fluid velocity, combined with the compactness of the macroscopic quantities associated with the kinetic distribution function, is sufficient to identify the limit of the drag force in the sense of distributions. Furthermore, one needs to identify the weak limit of the nonlinear stress tensor. In the absence of the kinetic-fluid coupling, this can be achieved by the $L^\infty$-truncation method developed in \cite{FMS00}. In the present coupled system, however, the truncation argument has to be combined with the compactness analysis of the drag force. One of the key points of our proof is to show that this framework remains stable in the presence of the kinetic-fluid coupling, thereby allowing us to identify the nonlinear stress tensor and close the global existence argument.

%
%
%
%
%

In what follows, we make the problem formulation and the main results precise. We begin by defining weak solutions to the coupled system \eqref{main_eq}--\eqref{init_data}. We then state the global existence result, followed by a large-time decay estimate under an a priori setting.

%
%
%
%
%
\subsection{Definition of weak solutions and main results}
Before presenting the precise statements of our results, we first introduce the relevant function spaces and the notion of weak solutions. 

We define
\[
\mathcal{H} := \left\{ \psi \in L^2(\T^3)\ :\ \nabla_x \cdot \psi = 0 \right\} \quad \mbox{and} \quad
V^{s,p} := \left\{ \psi \in \dot{W}^{s,p}(\T^3)\ : \ \nabla_x \cdot \psi = 0,\ \int_{\T^3} \psi\,dx = 0 \right\}.
\]

\begin{definition}[Weak solution]\label{def_weak}
Let $T > 0$. A pair $(f,u)$ is said to be a weak solution to \eqref{main_eq}--\eqref{init_data} on $[0,T)$ if:
\begin{enumerate}
\item $f \in L^\infty(0,T; L^1_+ \cap L^\infty(\T^3 \times \R^3))$,
\item $u \in L^\infty(0,T; \mathcal{H}) \cap L^p(0,T; \dot{W}^{1,p}(\T^3))$,
\item the kinetic equation holds in the weak sense: for all $\phi \in \mathcal{C}_c^\infty([0,T)\times \T^3 \times \R^3)$,
\[
\int_0^T \!\!\inttr f(\pa_t \phi + v \cdot \nabla_x \phi + (u-v)\cdot \nabla_v  \phi)\,dvdxdt = \inttr f_0 \phi(0,\cdot,\cdot)\,dxdv,
\]
\item and the fluid equation holds in the weak sense: for all divergence-free $\psi \in C_c^\infty([0,T)\times\T^3)$,
\begin{align*}
&\int_0^T \intt \left(-u \cdot \pa_t \psi +  (u\cdot \nabla_x )u \cdot \psi  + \left(\intr (u-v)\cdot \psi f\,dv\right)\right)dxds \\
&\qquad=  \intt u_0 \cdot \psi(0,\cdot)\,dx - \int_0^T \intt \tau(Du) : D\psi \,dxds.
\end{align*}

\end{enumerate}
\end{definition}

Our first main result establishes the global existence of such weak solutions.

\begin{theorem}[Global existence]\label{thm_main}
Let $T > 0$ and $p > \frac85$. Assume the initial data satisfy
\[
f_0 \in L^\infty(\T^3 \times \R^3), \quad (1+|v|^2)f_0 \in L^1(\T^3 \times \R^3), \quad \mbox{and} \quad u_0 \in \mathcal{H}.
\]
Then there exists a weak solution $(f,u)$ to the system \eqref{main_eq} on $[0,T)$ in the sense of Definition \ref{def_weak}, which satisfies the energy inequality:
\bq\label{en_ineq}
\begin{aligned}
&\sup_{t \in (0,T)}\left( \inttr \frac{|v|^2}{2} f\,dvdx + \intt \frac{|u|^2}{2}\,dx \right) + \int_0^T \intt \mu|D u|^p\,dxdt + \int_0^T \inttr |u-v|^2 f\,dvdxdt \\
&\quad \le \inttr \frac{|v|^2}{2} f_0\,dvdx + \intt \frac{|u_0|^2}{2}\,dx.
\end{aligned}
\eq
\end{theorem}
Our proof is based on a two-level approximation scheme combined with compactness and stability arguments. The first step involves mollifying the nonlinear terms and truncating the initial data in velocity space to construct a regularized system depending on a parameter $\e > 0$. For fixed $\e$, we apply a Galerkin approximation and establish uniform-in-$n$ estimates that allow passage to the limit as $n \to \infty$. The key observation here is that, although the fluid variable $u$ does not satisfy a conservation of momentum identity on the periodic domain, we may decompose $u$ as $u = w + u_c$, where $u_c$ captures the spatial average of $u$ and $w$ has zero average. This decomposition enables a priori estimates and compactness arguments to be carried out effectively in a periodic setting, with $w$ satisfying better integrability and embedding properties.

The second step consists in removing the regularization by sending $\e \to 0$. This requires careful analysis of the nonlinear drag term and the stress tensor. Strong compactness of macroscopic quantities is obtained using velocity averaging and moment estimates. The nonlinear stress term is treated via monotonicity and weak lower semicontinuity arguments. A notable feature of our approach is that the global-in-time existence result holds under the condition $p > \frac85$, which improves upon previous results requiring stronger coercivity assumptions on the stress tensor. Finally, the energy inequality is recovered in the limit using convexity and semicontinuity. However, the solutions constructed up to this point are shown to exist on a sufficiently small time interval, to facilitate the fixed point argument. Nevertheless, we can iterate the above process thanks to the energy inequality, thereby completing the proof of Theorem \ref{thm_main}.

\begin{remark}
The existence of weak solutions is not expected when $p$ is much less than $\frac 85$ due to difficulties in handling the nonlinear stress tensor. In the absence of kinetic coupling, one can construct measure-valued solutions for non-Newtonian fluids under the weaker assumption $p> \frac 65$. This has been shown, for instance, in \cite[Chapter 5, Theorem 2.17]{MNRR96}. However, for the coupled kinetic-fluid system \eqref{main_eq}, the drag term in the fluid equation imposes an additional restriction. As shown in \eqref{p_cond_coup} below, controlling the drag force requires $p > \frac{29 + \sqrt{91}}{25} \approx 1.541$. This indicates that constructing measure-valued solutions in the coupled setting would demand stronger assumptions than $p> \frac 65$.
\end{remark}

\begin{remark}
The existence of weak solutions can be easily extended to a more general form of the viscous stress tensor $\tau$. Specifically, it suffices to assume that $\tau$ continuously depends on its arguments and satisfies the following conditions:
\begin{itemize}
\item $($$p$-coercivity$)$: there exists $c_1 >0$, $\psi_1 \in L^1((0,T) \times \T^3)$ such that
$$\tau(t,x, K) : K  \geq c_1 |K|^p - \psi_1(t,x), \quad \forall \,  (t,x) \in (0,T) \times \T^3, \ K \in \R^{3\times 3}_{sym},$$
\item $($$(p-1)$-growth$)$:  there exists $c_2 >0$, $\psi_2 \in L^{\frac{p}{p-1}}((0,T) \times \T^3)$ such that
	$$| \tau(t,x,K) | \leq  c_2 |K|^{p-1} + \psi_2, \quad\forall \,  (t,x) \in (0,T) \times \T^3, \ K \in \R^{3\times 3}_{sym},$$
\item $($strict monotonicity$)$:
\bq\label{st.mt}
(\tau(t,x,K_1) - \tau(t,x,K_2))( K_1 - K_2 ) > 0, \quad \forall \,  (t,x) \in (0,T) \times \T^3, \ K_1, \, K_2 \in \R^{3\times 3}_{sym}, \,K_1 \neq K_2.
\eq
\end{itemize}
Further details and proofs of these general conditions can be found in \cite{FMS97,FMS00,MNRR96}.
\end{remark}

Now, we further investigate the large-time behavior of solutions under the additional assumption that the macroscopic particle density remains uniformly bounded. To measure convergence toward collective alignment, we introduce a {\it modulated energy functional} $\calE(t)$ defined by
\[
\calE(t) := \frac12\inttr |v - v_c|^2 f\,dxdv + \frac12\intt |u - u_c|^2 + \frac14|u_c - v_c|^2,
\]
where 
\[
v_c := \inttr vf\,dxdv \quad \mbox{and} \quad u_c:= \intt u\,dx
\]
denote the spatial averages of the particle and fluid velocities, respectively. This quantity effectively captures the deviation of the system from a monokinetic state in which both components are aligned in velocity.

Our second main result shows that this modulated energy decays to zero as time progresses, with an explicit rate depending on the power-law exponent $p$ characterizing the non-Newtonian fluid response. In particular, the case $p > 2$ leads to an algebraic decay, while $p \le 2$ yields an exponential decay.
\begin{theorem}[Large-time behavior]\label{thm_main2}
Let $\frac65 \leq p < \infty$ and let $(f,u)$ be a sufficiently regular solution \footnote{For example, we can assume that $(f,u)$ is a classical solution to \eqref{main_eq} with the energy balance, i.e. \eqref{en_ineq} with the inequality replaced by the equality.} to \eqref{main_eq}, and assume that the macroscopic particle density is uniformly bounded in time:
\bq\label{condi_large}
\int_{\R^3} f\,dv \in L^\infty(0,\infty; L^r(\T^3)) \quad  \mbox{for some }  r \in \left\{\begin{array}{lcl} 
 [\frac{6}{5p-6}, \infty] &\mbox{if}& \frac65 \leq p <2, \\[2mm]
 [\frac{3p}{5p-6}, \infty] &\mbox{if}& 2\le p<3, \\[2mm]
 (1, \infty] &\mbox{if} & p=3, \\[2mm]
 [1,\infty] & \mbox{if} & 3<p.   \end{array} \right.
\eq
Then the modulated energy $\mathcal{E}(t)$ satisfies the decay estimate
\[
\calE(t) 
\leq \begin{cases}
\left(\calE(0)^{\frac{2-p}{2}} + C_0 \left(\frac p2 - 1 \right) t \right)^{-\frac2{p-2}} & \text{if } p>2,\\[2mm]
\calE(0)e^{-C_0 t} & \text{if } \frac65 \le p \le 2,
\end{cases}
\]
for all $t \geq 0$, where $C_0 > 0$ is a constant independent of $t$.
\end{theorem}

\begin{remark}
The decay estimate in Theorem \ref{thm_main2} reflects the combined dissipative effects of the drag force and the non-Newtonian stress. When  $p > 2$, the energy decays algebraically, with slower rates for larger  $p$, corresponding to weaker diffusion in the fluid. In contrast, for  $\frac{6}{5} < p \le 2 $, the fluid dissipation is strong enough to yield exponential decay. We also notice that, since the spatial domain is bounded, one can simply assume 
\[
\int_{\R^3} f\,dv \in L^\infty(0,\infty; L^\infty(\T^3)).
\]
\end{remark}

\begin{remark}
The finite-mass condition on the particle density is a natural consequence of
mass conservation. Indeed, since $f$ is nonnegative, we have 
\[
\lt\|\intr f(t,\cdot,v)\,dv\rt\|_{L^1(\T^3)} = \inttr f(t,x,v)\,dxdv = \inttr f_0(x,v)\,dxdv.
\]
Thus, in the case $p>3$, condition \eqref{condi_large} is automatically satisfied by the conservation of mass.
\end{remark}

\begin{remark} 
As a consequence of Theorem \ref{thm_main2}, it follows from \cite[Remarks 1.5 \& 1.6]{Choi16} that the average velocities converge to a common limit:
\[
v_c(t), \ u_c(t) \to v^\infty \quad \mbox{as } t \to \infty,
\]
where
\[
v^\infty:= \frac12 \lt(\inttr vf_0\,dxdv + \intt u_0\,dx\rt).
\]
Moreover, the particle distribution converges in the bounded-Lipschitz distance to a monokinetic profile:
\[
{\rm d}_{BL} (f(t,x,v), \rho_f(t,x)\otimes \delta_{v^\infty}(v)) \to 0 \quad \mbox{as } t \to \infty,
\]
where ${\rm d}_{BL}$ denotes the bounded-Lipschitz distance for probability measures
\[
{\rm d}_{BL}(\mu, \nu) := \sup\lt\{ \lt|\inttr \phi \,d\mu -\inttr \phi\,d\nu\rt|  :  \|\phi\|_{L^\infty} \le 1, \    \sup_{(x,v) \neq (y,w)} \frac{|\phi(x,v) - \phi(y,w)|}{|(x,v) - (y,w)|} \le 1 \rt\},
\]
and  $\delta_{v^\infty}(v)$ denotes the Dirac mass at $v = v^\infty$. 
\end{remark}

To prove Theorem \ref{thm_main2}, we derive a differential identity for the modulated energy $\mathcal{E}(t)$, which quantifies the deviation of the coupled kinetic-fluid system from velocity alignment. The proof proceeds in two steps. First, we differentiate $\mathcal{E}(t)$ and show that it satisfies the identity 
\[
\frac{d}{dt} \mathcal{E}(t) + \mathcal{D}(t) = 0, 
\]
where $\mathcal{D}(t)$ denotes the total dissipation arising from fluid viscosity and the coupling term. Second, we estimate $\mathcal{E}(t)$ in terms of $\mathcal{D}(t)$ using Gagliardo--Nirenberg--Sobolev and Poincar\'e inequalities. This yields a differential inequality of the form
\[
\frac{d}{dt} \calE(t) + C_0 \calE(t)^{\gamma} \leq 0,
\]
with $\gamma = \frac{p}{2} > 1$ for $p > 2$, and $\gamma = 1$ when $p \le 2$, leading to the desired algebraic or exponential decay rates.

%
%
%
%
%
\subsection{Outline of the paper}
The rest of this paper is organized as follows. In Section \ref{sec_gw}, we prove the global existence of weak solutions by constructing approximate solutions and passing to the limit in two steps: first as the Galerkin dimension tends to infinity, and then as the mollification parameter vanishes, which results in a local-in-time weak solutions. Then we use the energy inequality to extend this local solution to any finite time interval. Section \ref{sec_lt} is devoted to the large-time behavior of regular solutions. Under a boundedness assumption on the particle density, we establish decay estimates for a modulated energy functional, with the rate depending on the power-law exponent of the fluid.

%
%
%
%
%
\section{Global-in-time existence of weak solutions}\label{sec_gw}
%
%
%
%
%
\subsection{Approximate system: mollification and Galerkin projection}
To construct global weak solutions to the system \eqref{main_eq}, we first introduce a regularized approximation that incorporates mollification and a Galerkin projection onto a finite-dimensional divergence-free subspace. This regularization plays a key role in controlling nonlinear and nonlocal terms, particularly the drag force and the convective term. In addition, regularizing the convective term facilitates passing to the limit in the nonlinear viscous stress.

We begin by regularizing the drag interaction and convective term in the momentum equation. For a fixed $\e > 0$, let $\theta$ be a standard mollifier:
\[
0 \leq \theta \in \mc^\infty_0(\T^3), \quad \mbox{supp}_x\theta \subseteq B(0,1), \quad \int_{\T^3} \theta(x)\,dx = 1,
\]
and we set a sequence of smooth mollifiers $\theta_\e(x) := \frac{1}{\e^3}\theta(\frac x\e)$.

For the fluid component, we employ a Galerkin approximation. Let $A := -\mathbb{P} \Delta$ denote the Stokes operator on the periodic domain $\T^3$, where $\mathbb{P}$ is the Leray projector onto the space $\mathcal{H}$ of divergence-free vector fields in $L^2(\T^3)$. The domain of $A$ is given by 
\[
D(A) := \lt\{ F \in [H^2(\T^3)]^3 \ \bigg| \ \nabla_x \cdot F = 0, \quad \intt F\,dx = 0\rt\}. 
\]
 Then we have a family of functions $ \mathcal{N}:= \{c_i\}_{i \ge 1}$ satisfying
\begin{enumerate}
\item
$\mathcal{N}$ is an orthonormal basis in $\mathcal{H}$.

\item
For every $i \in \mathbb{N}$, $c_i\in D(A)\cap C^\infty(\T^3)$ is an eigenfunction of the Stokes operator:
\[
Ac_i = \lambda_i c_i,
\]
with $0<\lambda_1 \le \lambda_2 \le \cdots \le \lambda_i \le \cdots$ and $\lambda_i \to \infty$ as $i \to \infty$. 
\end{enumerate}
For details, we refer to \cite[Theorem 2.24]{RRS}. Then we define the projection $P_n : \mathcal{H} \to \mathcal{H}_n :=$span$\{c_1, c_2, \dots, c_n\}$.

We now formulate the regularized approximate system. Let $f^{n,\e}$ and $w^{n,\e}$ denote the approximate distribution function and fluid velocity, respectively. 
The system reads:
\bq\label{reg_eq_k}
\int_0^T \!\!\inttr f^{n,\e}(\pa_t \phi + v \cdot \nabla_x \phi + (\theta_\e\star w^{n,\e}+u_c^{n,\e}-v )\cdot \nabla_v  \phi)\,dvdxdt = \inttr f_0 \phi(0,\cdot,\cdot)\,dxdv,
\eq
for all $\phi \in \mathcal{C}_c^\infty([0,T)\times \T^3 \times \R^3)$ and 
\bq\label{reg_eq_m}
\begin{aligned}
& \int_{\T^3} \pa_t w^{n,\e} c_i \,dx  +   \int_{\T^3} \bigg( ((\theta_\e \star w^{n,\e}) \cdot \nabla_x)w^{n,\e} \cdot c_i  +  \tau(Dw^{n,\e}) :D c_i  +  ( u_c^{n,\e} \cdot \nabla_x ) w^{n,\e} \cdot   c_i  \bigg) \,dx\\
& = 
\int_{\T^3} \bigg[  -\theta_\e \star \lt(\int_{\R^3} (\theta_\e \star w^{n,\e} +u_c^{n,\e}-v)f^{n,\e}\,dv\rt)  + \inttr(\theta_\e\star w^{n,\e}+u_c^{n,\e}-v)f^{n,\e}\,dxdv \bigg] c_i \,dx  
\end{aligned}
\eq
for $i=1,\dots,n$, with 
\[
w^{n,\e} (t,x) = \sum_{i=1}^n \alpha_i^\e(t) c_i(x)
\] 
subject to regularized initial data
\bq\label{reg_init}
\begin{aligned}
&f^{n,\e}(0,x,v)= f_0^{n,\e}(x,v) := f_0(x,v)\mathds{1}_\e (v), \quad (x,v) \in \T^3 \times \R^3 \quad \mbox{and} \\
&  w_0^{n,\e} := w^{n,\e}(0,x)= P_n (\theta_\e \star w_0) := P_n \lt(\theta_\e\star u_0 - \intt u_0\,dx\rt),  \quad x \in \T^3,
\end{aligned}
\eq
where the velocity cutoff $\mathds{1}_\e(v)$ is defined as
\[
\mathds{1}_\e (v) := \lt\{\begin{array}{lcl}1 & |v|\le \frac1\e,\\[2mm]
0 & \mbox{otherwise}. \end{array}\rt.
\]
The term $u_c^{n,\e}(t)$ denotes the average component of the fluid velocity:
\[\begin{aligned}
u_c^{n,\e}(t):= \lt(\intt u_0\,dx\rt)e^{-\int_0^t \intt \rho_{f^{n,\e}}\,dxds }  - \int_0^t e^{-\int_s^t \intt \rho_{f^{n,\e}}\,dxdr }\intt (\rho_{f^{n,\e}} (\theta_\e\star  w^{n,\e}) - m_{f^{n,\e}})\,dxds, \quad 
\end{aligned}\]
which solves
\[\begin{aligned}
u_c^{n,\e} (t) &=\intt u_0(x)\,dx - \int_0^t \intt \lt[\rho_{f^{n,\e}}(\theta_\e \star w^{n,\e}  + u_c^{n,\e})  -m_{f^{n,\e}}\rt]\,dxds\\
&= \intt u_0(x)\,dx - \int_0^t \inttr \lt[\theta_\e \star w^{n,\e} + u_c^{n,\e}  -v\rt]f^{n,\e}\,dxdvds,
\end{aligned}\]
where
\[
\rho_{f^{n,\e}} := \intr f^{n,\e} \,dv, \quad m_{f^{n,\e}} := \intr vf^{n,\e}\,dv.
\]

The initial data are smooth and satisfy uniform bounds: $f_0^{n,\e}$ is compactly supported in $v$ and belongs to $(L^1 \cap L^\infty)(\T^3 \times \R^3)$, while $w_0^{n,\e} \in H^\beta(\T^3)$ for any $\beta > 0$ with $\nabla_x \cdot w_0^{n,\e} = 0$.
 
A few comments on the structure of the regularized system \eqref{reg_eq_k}--\eqref{reg_eq_m} are in order. The kinetic equation involves the mollified fluid velocity $\theta_\e \star w^{n,\e}$ to ensure sufficient smoothness of the transport field. In principle, since $w^{n,\e}$ lies in a finite-dimensional space, the kinetic transport equation is well-posed even without this mollification. However, mollifying the fluid velocity in the kinetic part provides better stability and consistency with the regularization in the fluid equation.

In the Non-Newtonian fluid part, the fluid velocity appearing in the drag force $\int_{\R^3} (\theta_\e \star w^{n,\e} +u_c^{n,\e}-v)f^{n,\e}\,dv$ is mollified once, and the entire drag force term is mollified again as $\theta_\e \star (\cdot)$. This double mollification is crucial for establishing the uniform energy inequality in Lemma \ref{unif_eps_est} below. If only the velocity field inside the drag force were mollified, controlling the nonlinear coupling between fluid and kinetic components would become significantly more delicate. The second mollification provides compatibility between the fluid and kinetic dissipations, allowing us to derive uniform energy estimates independent of $\varepsilon$ and $n$.

Additional regularization by convolution in the convective term is included to handle lower values of the exponent $p$. This allows the stability of the approximation for the nonlinear term to be treated via the so-called truncation (or capacity) method; see \cite{FMS97,FMS00}. That approach relies critically on the strict monotonicity of $\tau$ and is based on the construction of special divergence-free truncated test functions.
 
%
%
%
%
%
\subsection{Decoupling and fixed point formulation}
To construct solutions to the regularized system \eqref{reg_eq_k}--\eqref{reg_eq_m}, we employ a fixed point argument based on Schauder's fixed point theorem. The key idea is to decouple the system by freezing certain nonlinear terms and then constructing a solution map that returns to the same functional space. For the proper application of this strategy, the size of the time interval where weak solutions to \eqref{reg_eq_k}--\eqref{reg_eq_m} are constructed should be sufficiently small. Based on the energy inequality, this local-in-time solution will be extended to any finite time interval.

We begin by defining the Banach space
\[
E:= L^{4p'}(0,T;L^2(\T^3))  \times  L^{2p'}(0,T;L^2(\T^3))  \times \lt(L^{4p'}(0,T;V^{0,2}) \cap L^{4p'}(0,T;\mathcal{H}_n)\rt),
\]
where $p'$ denotes the H\"older conjugate of $p$, i.e.,
\[
\frac1p+\frac1{p'}=1.
\]
The space $E$ is equipped with the norm
\[
\|(\bar \rho, \bar m, \bar w) \|_E := \|\bar\rho\|_{L^{4p'}(0,T; L^2(\T^3))} + \|\bar m\|_{L^{2p'}(0,T; L^2 (\T^3))} + \|\bar w \|_{L^{4p'}(0,T;V^{0,2})}.
\]

For a given triple $(\bar\rho, \bar m, \bar w) \in E$, we consider the following decoupled system:  
\begin{equation}\label{dec_eq_k}
\int_0^T \!\!\inttr f^{n,\e}(\pa_t \phi + v \cdot \nabla_x \phi + (\theta_\e\star \bar{w}+u_c^{n,\e}-v )\cdot \nabla_v  \phi)\,dvdxdt = \inttr f_0^{n,\e} \phi(0,\cdot,\cdot)\,dxdv,
\end{equation}
for all $\phi \in \mathcal{C}_c^\infty([0,T)\times \T^3 \times \R^3)$ and 
\bq\label{dec_eq_m}
\begin{aligned}
&  \frac{d}{dt} \int_{\T^3} w^{n,\e} \cdot c_i \,dx  =   -\int_{\T^3} \bigg( ((\theta_\e \star w^{n,\e}) \cdot \nabla_x)w^{n,\e} \cdot c_i  +  \tau(Dw^{n,\e}) : D c_i  +  ( u_c^{n,\e} \cdot \nabla_x ) w^{n,\e} \cdot  c_i  \bigg) \,dx
\\ & + 
\int_{\T^3} \bigg[  -\theta_\e \star \lt[ \bar{\rho}(\theta_\e \star \bar{w} +u_c^{n,\e})-\bar{m}\rt]  + \int_{\T^3}(\bar{\rho}(\theta_\e\star \bar{w}+u_c^{n,\e})-\bar{m})\,dx \bigg] \cdot c_i \,dx \quad \mbox{ for } i=1,\dots,n
\end{aligned}
\eq
with
\[
u_c^{n,\e}(t):= \lt(\intt u_0\,dx\rt)e^{-\int_0^t \intt \bar \rho\,dxds }- \int_0^t e^{-\int_s^t \intt \bar \rho \,dxdr }\intt (\bar \rho (\theta_\e\star \bar w) - \bar m)\,dxds.
\]

This system is linear in $f^{n,\e}$ and semilinear in $w^{n,\e}$ for fixed data $(\bar\rho, \bar m, \bar w)$. Hence, for each such triple, we can (under suitable assumptions) construct a unique solution $(f^{n,\e}, w^{n,\e})$.

We then define the solution operator $\mathcal{T}: E \to E$ given as
\[
\mathcal{T}(\bar\rho, \bar m, \bar w) =(\mathcal{T}_1(\bar\rho, \bar m, \bar w), \mathcal{T}_2(\bar\rho, \bar m, \bar w), \mathcal{T}_3(\bar\rho, \bar m, \bar w))  :=(\rho_{f^{n,\e}}, m_{f^{n,\e}}, w^{n,\e}).
\]
The operator $\mathcal{T} : E \to E$ thus maps a given triplet of macroscopic quantities to the updated ones obtained from the solution of the decoupled system. Our goal is to prove that $\mathcal{T}$ admits a fixed point in $E$, which then yields a solution to the full regularized system \eqref{reg_eq_k}--\eqref{reg_eq_m} via Schauder's fixed point theorem.
%
%
%
%
%
\subsection{Existence of the regularized system \eqref{reg_eq_k}--\eqref{reg_eq_m}} In this part, we provide the global-in-time existence of weak solutions to the regularized system \eqref{reg_eq_k}--\eqref{reg_eq_m}.

%
%
%
%
%
\subsubsection{The operator $\mt$ is well-defined} 
We show that the operator $\mathcal{T}$ defined in the previous subsection is well-defined on $E$. A key tool is the following regularity lemma, which provides integrability estimates for the macroscopic density and momentum derived from a kinetic distribution function.
\begin{proposition}[{\cite[Proposition A.1]{CJ21}}]\label{PA.1}
Let $f = f(t,x,v)$ satisfy
\[
\|f\|_{L^\infty((0,T)\times \T^3 \times \R^3)}<\infty,
\]
and let $d_0 \in \N$ such that
\[
\sup_{t > 0}\iint_{\T^3 \times \R^3} (1+|v|)^\ell f(t,x,v)\,dxdv \le \infty, \quad \forall t \in (0,T), \quad \forall \, \ell \in [0,d_0].  
\]
Then, there exists a constant $C>0$ which depends on $\ell$ but independent of $f$ such that for all $t \in [0,T]$,
\[
\sup_{t >0}\|\rho_f (t)\|_{L^q} \le C\|f\|_{L^\infty((0,T)\times \T^3 \times \R^3)}^{1-\frac1q} \cdot \lt(\sup_{t > 0}\iint_{\T^3 \times \R^3} (1+|v|)^{d_0} f(t,x,v)\,dxdv\rt)^{\frac 1q}, \quad \forall \,q \in \left[1, \frac{d_0+3}3 \right)
\]
and
\[
 \sup_{t >0}\|m_f(t)\|_{L^q} \le C\|f\|_{L^\infty((0,T)\times \T^3 \times \R^3)}^{1-\frac1q} \cdot \lt(\sup_{t > 0}\iint_{\T^3 \times \R^3} (1+|v|)^{d_0} f(t,x,v)\,dxdv\rt)^{\frac 1q}, \quad \forall \, q \in \left[1, \frac{d_0+3}4 \right).
\]
\end{proposition}

Now we show that $\mathcal{T}$ is well-defined, which is summarized in the following lemma.
\begin{lemma}\label{dec_bdd}
Let $T > 0$, $R>0$, and $p > \frac85$. For any $(\bar\rho, \bar m, \bar w) \in E$ with $\|(\bar \rho, \bar m, \bar w)\|_{E} <R$, there exists a unique solution $(f^{n,\e}, w^{n,\e})$ to \eqref{dec_eq_k}--\eqref{dec_eq_m} with initial data $(f_0^{n,\e}, w_0^{n,\e})$, and the pair satisfies
\[\begin{aligned}
&\|f^{n,\e}(t)\|_{L^r(\T^3 \times \R^3)} \le e^{3\frac{(r-1)}{r} t} \|f_0^{n,\e}\|_{L^r(\T^3 \times \R^3)} ,\\
&\sup_{t \in (0,T)} \inttr |v|^\ell f^{n,\e}\,dvdx \le  C_1 (1+ C_2 T^{\frac{3p}{4p+1}}), 
\end{aligned}\]
for any $r \in [1, \infty]$ and finite $\ell \in \N$,  where $C_1= C_1(\e, \ell, \|(1+|v|)^\ell f_0^{n,\e}\|_{L^1(\T^3\times\R^3)})$ and $C_2 = C_2(\e, \ell, \|(1+|v|)^\ell f_0^{n,\e}\|_{L^1(\T^3\times\R^3)}, \|u_0\|_{L^2(\T^3)}, T, R)$ are positive constants while $C_2$ remains bounded as $T \to 0$. Moreover,
\[
\|w^{n,\e}\|_{L^\infty(0,T;L^2(\T^3))\cap L^p(0,T;V^{1,p})}^2 +  \|\pa_t w^{n,\e}\|_{L^{p'}(0,T;\mathcal{Y}^*)}^2  \le C_4(\|w_0^{n,\e}\|_{L^2(\T^3)})(1+ C_5 T^\zeta),
\]
where $\mathcal{Y} := L^p(0,T;V^{1,p})\cap L^{s'}((0,T)\times\T^3)$, $s= \frac{5p}{8}>1$, $\zeta=\zeta(p)>0$ is a positive constant depending only on $p$, and $C_4 = C_4(\|w_0^{n,\e}\|_{L^2(\T^3)})$ and $C_5 = C_5(\e, \|u_0\|_{L^2(\T^3)}, T, R)$ are positive constants while $C_5$ remains bounded as $T\to 0$.
\end{lemma}
\begin{proof}
The kinetic equation is a linear transport equation with smooth coefficients; thus, existence and uniqueness follow from the method of characteristics (see \cite{DL89}). Standard arguments and their slight modifications yield $L^r$ and velocity moment estimates (see \cite[Lemma 4.1]{CCK16}).  Here, for readers' convenience, we present the velocity moment estimates. First, from \eqref{dec_eq_k} we obtain
\[\begin{aligned}
\frac{d}{dt}\inttr |v| f^{n,\e}\,dxdv &= \inttr \frac{v}{|v|}\cdot (\theta_\e \star \bar w + u_c^{n,\e} - v) f^{n,\e}\,dxdv\\
&\le \lt(\|\theta_\e \star \bar w\|_{L^\infty(\T^3)} + |u_c^{n,\e}| \rt) \|f^{n,\e}\|_{L^1(\T^3\times\R^3)}\\
&\le \lt(\|\theta_\e \star \bar w\|_{L^\infty(\T^3)} + |u_c^{n,\e}| \rt) \|f_0^{n,\e}\|_{L^1(\T^3\times\R^3)}.
\end{aligned}\]
Let us estimate now the average component $u_c^{n,\e}$. From its explicit formula, we have
\[\begin{aligned}
|u^{n,\e}_c(t)| &\le \lt(\|u_0\|_{L^1(\T^3)} + \int_0^t \| \bar \rho(\theta_\e \star \bar w) - \bar m\|_{L^1(\T^3)}\,ds \rt)\exp\lt( \int_0^t \|\bar \rho\|_{L^1(\T^3)}\,ds\rt)\\
&\le \lt( \|u_0\|_{L^1(\T^3)} + \int_0^t \lt(\|\bar \rho\|_{L^2(\T^3)} \|\bar w\|_{L^2(\T^3)} + \|\bar m\|_{L^2(\T^3)}\rt)\,ds\rt)e^{T^{\frac{3p+1}{4p}} \|(\bar \rho, \bar m, \bar w)\|_{E}}\\
&\le  \lt( \|u_0\|_{L^1(\T^3)} + T^{\frac{p+1}{2p}}(1+\|(\bar \rho, \bar m, \bar w)\|_{E})^2 \rt)e^{T^{\frac{3p+1}{4p}} \|(\bar \rho, \bar m, \bar w)\|_{E}}\\
&\le C(\|u_0\|_{L^2(\T^3)}, T, R),
\end{aligned}\]
where $C=C(\|u_0\|_{L^2(\T^3)}, T, R)$ is a positive and remains bounded as $T \to 0$. Together with 
\[
\|\theta_\e \star\bar w\|_{L^\infty(\T^3)} \le C(\e) \|\bar w\|_{L^2(\T^3)},
\]
we get
\[\begin{aligned}
\inttr |v| f^{n,\e}\,dxdv &\le \inttr |v| f_0^{n,\e}\,dxdv + C(\e) \|f_0^{n,\e}\|_{L^1(\T^3\times\R^3)}\int_0^t \|\bar w\|_{L^2(\T^3)}\,ds\\
&\quad + T C_1(\|u_0\|_{L^2(\T^3)}, T, R)\|f_0^{n,\e}\|_{L^1(\T^3\times\R^3)}\\
&\le C_1 (1+ C_2 T^{\frac{3p}{4p+1}}),
\end{aligned}
\]
where $C_1= C_1(\e, \|(1+|v|)f_0^{n,\e}\|_{L^1(\T^3\times\R^3)})$ and $C_2:=C_2(\|u_0\|_{L^2(\T^3)}, T, R))$ are positive constants and $C_2$ remains bounded as $T \to 0$.

Then inductively, for $\ell \in \N$,
\[\begin{aligned}
\frac{d}{dt}& \inttr \frac{|v|^\ell}{\ell} f^{n,\e}\,dxdv \\
&= \inttr |v|^{\ell-2}v \cdot (\theta_\e \star \bar w + u_c^{n,\e} - v) f^{n,\e}\,dxdv\\
&\le \lt(\|\theta_\e \star \bar w\|_{L^\infty(\T^3)} + |u_c^{n,\e}| \rt) \inttr |v|^{\ell-1}f^{n,\e}\,dxdv\\
&\le C_1(\e, \ell-1, \|(1+|v|)^{\ell-1}f_0^{n,\e}\|_{L^1(\T^3\times\R^3)}, \|u_0\|_{L^2(\T^3)}, T, R)(1+\|\bar w\|_{L^2(\T^3)})
\end{aligned}\]
and hence, we also get
\[\begin{aligned}
&\inttr |v|^\ell f^{n,\e}\,dxdv \\
&\quad\le \inttr |v|^\ell f_0^{n,\e}\,dxdv+ TC_1(\e, \ell, \|f_0^{n,\e}\|_{L^1(\T^3\times\R^3)}, \|u_0\|_{L^2(\T^3)}, T, R)\\
&\qquad + T^{\frac{3p+1}{4p}}C_1(\e, \ell, \|f_0^{n,\e}\|_{L^1(\T^3\times\R^3)}, \|u_0\|_{L^2(\T^3)}, T, R)\\
&\quad\le C_1 (1+ C_2 T^{\frac{3p}{4p+1}})
\end{aligned}\]
where $C_1= C_1(\e, \ell, \|(1+|v|)^\ell f_0^{n,\e}\|_{L^1(\T^3\times\R^3)})$ and $C_2 = C_2(\e, \ell, \|(1+|v|)^\ell f_0^{n,\e}\|_{L^1(\T^3\times\R^3)}, \|u_0\|_{L^2(\T^3)}, T, R)$ are positive constants and $C_2$ remains bounded as $T \to 0$, completing the induction.

For the fluid equation, let us write the forcing term as
\[
F^n := - \theta_\e \star [\bar\rho(\theta_\e \star \bar w +u_c^{n,\e})-\bar m] + \intt[\bar\rho(\theta_\e \star \bar w+u_c^{n,\e})-\bar m]\,dx.
\]
This is a time-dependent external force in a finite-dimensional Galerkin system. Existence and uniqueness follow from ODE theory. Moreover, we easily obtain
\[
\frac12\frac{d}{dt}\|w^{n,\e}\|_{L^2}^2 + c \|w^{n,\e}\|_{V^{1,p}}^p \le \intt w^{n,\e} \cdot F^n\,dx, \quad \frac{d}{dt}\intt w^{n,\e}(t)\,dx = 0. 
\]
Since
\[\begin{aligned}
\|F^n\|_{L^{p'}(\T^3)} &\le C(\e) \| \bar\rho (\theta_\e \star \bar w  + u_c^{n,\e}) - \bar m\|_{L^1(\T^3)} \\
&\le C(\e) \lt(\|\bar \rho\|_{L^2(\T^3)} (\|\bar w\|_{L^2(\T^3)} + |u_c^{n,\e}|) + \|\bar m\|_{L^2(\T^3)}\rt),
\end{aligned}\]
the Young inequality and the Poincar\'e inequality give
\[
\frac{d}{dt} \|w^{n,\e}\|_{L^2(\T^3)}^2 + c \|w^{n,\e}\|_{V^{1,p}}^p \le C\|F^n\|_{L^{p'}(\T^3)}^{p'}.
\]
Hence, we have
\[\begin{aligned}
\|&w^{n,\e}\|_{L^2(\T^3)}^2 + c\int_0^t \|w^{n,\e}\|_{V^{1,p}}^p\,ds \\
&\le \|w_0^{n,\e}\|_{L^2(\T^3)}^2 + C\int_0^t \|F^n\|_{L^{p'}(\T^3)}^{p'}\,ds\\
&\le \|w_0^{n,\e}\|_{L^2(\T^3)}^2 + C(\e) \lt( T^{1/2}R^2 + T^{\frac{3p+1}{4p}}C_1(\|u_0\|_{L^2(\T^3)}, T,R)\rt)\\
&\le \|w_0^{n,\e}\|_{L^2(\T^3)}^2 + C_3T^{1/2},
\end{aligned}\]
where $C_3=C_3(\e,\|u_0\|_{L^2}, T, R)$ is a positive and remains bounded as $T\to 0$. 
This implies the desired bound for $w^{n,\e}$ in $L^\infty(0,T;V^{0,2})\cap L^p(0,T;V^{1,p})$. By the Gagliardo--Nirenberg inequality, we also obtain $w^{n,\e} \in L^{\frac{5p}{3}}((0,T)\times\T^3)$. Since the convolution is taken with a fixed mollifier $\theta_\e$, it follows that $(\theta_\e \star w^{n,\e})\cdot \nabla_x w^{n,\e} \in L^s((0,T)\times \T^3)$:
\bq\label{w_est_sp}
\begin{aligned}
\|w^{n,\e}\|_{L^{\frac{5p}{3}}((0,T)\times\T^3) } &\le C\|\nabla_x w^{n,\e}\|_{L^p((0,T)\times\T^3)}^{\frac35} \|w^{n,\e}\|_{L^\infty(0,T;L^2(\T^3))}^{\frac25} \\
&\le C\lt(  \|w_0^{n,\e}\|_{L^2(\T^3)}^2 + C_3T^{1/2}\rt)^{\frac{p+3}{5p}},\\
 \|((\theta_\e \star w^{n,\e})\cdot \nabla_x ) w^{n,\e}\|_{L^s((0,T)\times\T^3)}&\le \|\nabla_x w^{n,\e}\|_{L^p((0,T)\times\T^3)}\|w^{n,\e}\|_{L^{\frac{5p}{3}}((0,T)\times\T^3) }\\
&\le  C\lt(\|w_0^{n,\e}\|_{L^2(\T^3)}^2 + T^{1/2} C_1\lt(\e, \|u_0\|_{L^2(\T^3)}, T,R\rt)\rt)^{\frac{p+8}{5p}}.
\end{aligned}\eq
Concerning the  bound for $\pa_t w^{n,\e}$, choose $\phi \in \mathcal{Y}$ to deduce
\begin{align*}
\lt|\int_0^T \intt \pa_t w^{n,\e} \cdot \phi \,dxdt \rt| &\le  \|( (\theta_\e \star w^{n,\e})  \cdot \nabla_x) w^{n,\e}\|_{L^s( (0,T)\times \T^3)}
\| \phi\|_{L^{s'}((0,T)\times \T^3)} \\
&\quad + C\int_0^T \|\nabla_x \phi(t)\|_{L^p(\T^3)}\|(1 + |Dw^{n,\e}|)\|_{L^p(\T^3)}^{p-1}\,dt \\
&\quad + \|(u_c^{n,\e} \cdot \nabla_x ) w^{n,\e}\|_{L^{s}((0,T)\times\T^3))}\| \phi\|_{L^{s'}((0,T)\times \T^3)} \\
&\quad + \| F^n \|_{L^{p'}((0,T)\times \T^3)} \|\phi\|_{L^p((0,T)\times \T^3)} \\
&\le C\lt(\|w_0^{n,\e}\|_{L^2(\T^3)}^2 + C_3 T^{1/2} \rt)^{\frac{p+8}{5p}}\|\phi\|_{L^{s'}((0,T)\times\T^3)}\\
&\quad + C( 1+ \|\nabla_x w^{n,\e}\|_{L^p((0,T)\times \T^3)}^p)\|\phi\|_{L^p(0,T;V^{1,p})}\\
&\quad +C \lt(  \|w_0^{n,\e}\|_{L^2(\T^3)}^2 + C_3T^{1/2}  \rt)^{\frac{11}{5p}} \| \phi\|_{L^{s'}((0,T)\times \T^3)}   \\
&\quad + T^{\frac{p-1}{2p}}C_1(\e,\|u_0\|_{L^2(\T^3)}, T, R) \|\phi\|_{L^p(0,T;V^{1,p})}\\
&\le C_4(1+ C_5T^\zeta ),
\end{align*}
where $\zeta = \min\lt\{ \frac{p+8}{10p}, \frac{p-1}{2p}, \frac{11}{10p}\rt\}$, $C_4 = C_4(\|w_0^{n,\e}\|_{L^2(\T^3)})$ and $C_5 = C_5(\e, \|u_0\|_{L^2(\T^3)}, T, R)$ are positive constants while $C_5$ remains bounded as $T\to 0$, and we used
\[\begin{aligned}
\|(u_c^{n,\e}\cdot \nabla_x) w^{n,\e}\|_{L^s((0,T)\times\T^3)} &\le \|\nabla_x w^{n,\e}\|_{L^p((0,T)\times\T^3)}\|u_c^{n,\e}\|_{L^{\frac{5p}{3}}(0,T)}\\
&\le \lt(  \|w_0^{n,\e}\|_{L^2(\T^3)} + T^{1/4} C_1\lt(\e, \|u_0\|_{L^2(\T^3)}, T,R)\rt) \rt)^{\frac 2p} \cdot T^{\frac {3}{5p}} C_1(\|u_0\|_{L^2(\T^3)}, T,R)\\
&\le C\lt(  \|w_0^{n,\e}\|_{L^2(\T^3)}^2 + \max\{C_1, C_3\}T^{1/2}  \rt)^{\frac{11}{5p}}. 
\end{aligned}\]
Thus, we have the desired bound for $\pa_t w^{n,\e}$ and this completes the proof.
\end{proof}

%
%
%
%
%
\subsubsection{Compactness of the solution map}
To apply Schauder's fixed point theorem, we must show that $\mathcal{T}$ is compact. For this, we use the following variant of the velocity averaging lemma.
\begin{lemma}[{\cite[Lemma 2.7]{KMT13}}]\label{vel_av}
For $p\in (1,\infty]$, let $\{f^m\}_{m\in \N}$ be compactly supported in velocity uniformly in $m$ and $\{G^m\}_{m\in \N}$ be bounded in $L_{loc}^p([0,T]\times\T^3\times\R^3)$. Suppose that $f^m$ and $G^m$ satisfy
\[
\pa_t f^m + v \cdot \nabla_x f^m = \nabla_v^\ell G^m, \quad f^m|_{t=0} = f_0 \in L^p(\T^3 \times \R^3)
\] 
for some multi-index $\ell$ and 
\[
\sup_{m\in\N}\lt(\|f^m\|_{L^\infty((0,T)\times \T^3\times\R^3)} +  \||v|^\beta f^m\|_{L^\infty(0,T; L^1(\T^3\times\R^3))}\rt)<\infty,
\]
for some $\beta>1$. Then the sequences
\[
\lt\{ \intr f^m  \,dv\rt\}_{m\in \N} , \quad \lt\{ \intr vf^m  \,dv\rt\}_{m\in \N} 
\]
are relatively compact in $L^{q_1}((0,T)\times\R^3)$ and  $L^{q_2}((0,T)\times\R^3)$ for $q_1 \in \lt(1, \frac{3+\beta}{3}\rt)$ and $q_2 \in \lt(1, \frac{3+\beta}{4}\rt)$, respectively.
\end{lemma}
Then, we have the following.
\begin{lemma}
Let $T > 0$, $p > 8/5$, and let $\{(\bar\rho_k, \bar m_k, \bar w_k)\}_k$ be a uniformly bounded sequence in $E$. Then, up to a subsequence, $\{\mathcal{T}(\bar\rho_k, \bar m_k, \bar w_k)\}_k$ converges strongly in $E$.
\end{lemma}
\begin{proof}
By Lemma \ref{dec_bdd}, the sequence $\{w_k := \mathcal{T}_3(\bar\rho_k, \bar m_k, \bar w_k)\}_k$ is uniformly bounded in $L^\infty(0,T;V^{0,2}) \cap L^p(0,T;V^{1,p})$ with bounded time derivative in $L^{p'}(0,T;\mathcal{Y}^*)$. By the Aubin--Lions lemma, we conclude that $\{w_k\}_k$ is relatively compact in $L^p(0,T;V^{0,2})$. Moreover , since $\{w_k\}_k$ is uniformly bounded in $L^\infty(0,T;V^{0,2})$, they are relatively compact in $L^{4p'}(0,T;V^{0,2})$. 

For the kinetic components, the uniform bounds on $f_k := f^{n,\e}_k$ and its velocity moments allow us to apply the velocity averaging lemma (Lemma \ref{vel_av}), yielding strong compactness of $\rho_{f_k}$ and $m_{f_k}$. Hence $\mathcal{T}$ is compact.
\end{proof}

The above estimates and compactness properties allow us to employ Schauder's fixed point theorem and obtain the existence of a fixed point:
\begin{lemma}\label{unif_eps_est}
Let $T > 0$ and $p > \frac85$. Then there exists a solution $(f^{n,\e}, w^{n,\e})$ to the regularized system \eqref{reg_eq_k}--\eqref{reg_init}, which satisfies the following estimates:
\bq\label{unif_est_soln}
\begin{aligned}
&\sup_{t \in (0,T)}\|f^{n,\e}(t)\|_{L^r(\T^3 \times \R^3)} \le e^{3T}\|f_0^{n,\e}\|_{L^r(\T^3 \times \R^3)} \quad \mbox{ for } r\in [1,\infty],\\
&\sup_{t \in (0,T)}\lt(\inttr \frac{|v|^2}{2} f^{n,\e}\,dvdx + \intt\frac{|w^{n,\e}|^2}{2}\,dx + \frac{|u_c^{n,\e}|^2}{2}\rt) + \int_0^T\intt \mu |D w^{n,\e}|^p\,dxdt \\
&\qquad+ \int_0^T \inttr |\theta_\e \star w^{n,\e}+u_c^{n,\e}-v|^2 f^{n,\e}\,dvdxdt\\
& \le \lt(\inttr \frac{|v|^2}{2} f_0^\e\,dvdx + \intt\frac{|w_0^{n,\e}|^2}{2}\,dx + \frac{|u_c^{n,\e}(0)|^2}{2}\rt).
\end{aligned}
\eq
\end{lemma}
\begin{proof}
We first construct a solution on a sufficiently small time interval. From Lemma \ref{dec_bdd} and Proposition \ref{PA.1}, we first observe that,
for fixed $n$ and $\e$, the image of $\mathcal T$ satisfies
\[
\|\rho_{f^{n,\e}}\|_{L^{4p'}(0,T;L^2(\T^3))} + \|m_{f^{n,\e}}\|_{L^{2p'}(0,T;L^2(\T^3))} + \|w^{n,\e}\|_{L^{4p'}(0,T;L^2(\T^3))} \le C_0 T^\alpha + C_0 T^\alpha C_T,
\]
for some $\alpha>0$, where $C_0$ depends on the regularized initial data,
$n$, and $\e$, and $C_T$ remains bounded as $T\to0$. Hence, by choosing
$R>0$ sufficiently large and then choosing $T_*>0$ sufficiently small, we have
\[
\|(\bar \rho,\bar m,\bar w)\|_E\le \frac R2 \quad\Longrightarrow\quad \|\mathcal T(\bar \rho,\bar m,\bar w)\|_E\le \frac R2.
\]
Let
\[
\mathcal B_R:=\lt\{(\bar \rho,\bar m,\bar w)\in E: \|(\bar \rho,\bar m,\bar w)\|_E\le \frac R2\rt\}.
\]
Then $\mathcal B_R$ is a nonempty, closed, bounded, and convex subset of $E$, and $\mathcal T$ maps $\mathcal B_R$ into itself.

Moreover, the stability estimates for the decoupled kinetic equation and the finite-dimensional Galerkin system imply that $\mathcal T$ is continuous on $\mathcal B_R$. Together with the compactness of $\mathcal T$ proved above, Schauder's fixed point theorem yields a fixed point of $\mathcal T$ in $\mathcal B_R$. This fixed point gives a solution $(f^{n,\e},w^{n,\e})$ to \eqref{reg_eq_k}--\eqref{reg_init} on $[0,T_*)$.

The $L^r$ estimate for $f^{n,\e}$ easily follows from the standard estimate. The energy inequality is obtained by testing the kinetic and fluid equations with $|v|^2$ and $w^{n,\e}$, respectively, and using the definition of the forcing term:
\[
\intt  F^n \cdot w^{n,\e} \,dx = -\inttr (\theta_\e \star w^{n,\e} +u_c^{n,\e}-v)\cdot (\theta_\e \star  w^{n,\e})f^{n,\e}\,dvdx
\]
In addition, the evolution of the average velocity component satisfies
\[
\frac{|u_c^{n,\e}(t)|^2}{2} = \frac{|u_c^{n,\e}(0)|^2}{2} - \int_0^t \inttr (\theta_\e \star w^{n,\e} + u_c^{n,\e} - v)\cdot u_c^{n,\e} f^{n,\e} \,dxdvds
\]
which accounts for the corresponding term in the total energy balance. Combining all the estimates yields the energy estimates up to $t=T_*$. 

It remains to extend the solution to an arbitrary finite time interval $[0,T)$. The energy inequality obtained above gives bounds which depend only on the initial energy and the fixed parameters $n$ and $\e$, and hence are independent of the number of iterations. Thus, the length of the local existence interval can be chosen uniformly along the iteration. As in \cite{BDGM09}, we restart the local construction at $T_*-\delta$, with $\delta>0$ sufficiently small, and extend the solution to $[0,2T_*-\delta)$. Repeating this overlapping continuation procedure finitely many times yields a solution on the whole interval $[0,T)$. 
\end{proof}

%
%
%
%
%
\subsection{Proof of Theorem \ref{thm_main}} In this section, we establish the global-in-time existence of weak solutions to the system \eqref{main_eq}--\eqref{init_data}. Our strategy is to construct approximate solutions and perform two successive limiting procedures. We first pass to the limit $n \to \infty$ in the Galerkin approximation of a regularized system for fixed $\e > 0$, and then send $\e \to 0$ to recover a weak solution to the original system.
%
%
%
%
%
\subsubsection{The limit $n \to \infty$ for fixed $\e$}

Now our aim is to pass to the limit $n \to \infty$ for fixed $\e > 0$ in the regularized system \eqref{reg_eq_k}--\eqref{reg_init},  thereby establishing the existence of weak solution $(f^\e, w^\e)$ satisfying
\bq\label{reg_eq_e_k}
\int_0^T \!\!\inttr f^{\e}(\pa_t \phi + v \cdot \nabla_x \phi + (\theta_\e\star w^{\e}+u_c^{\e}-v )\cdot \nabla_v  \phi)\,dvdxdt = \inttr f^\e_0 \phi(0,\cdot,\cdot)\,dxdv,
\eq
holds for all $\phi \in \mathcal{C}_c^\infty([0,T)\times \T^3 \times \R^3)$ and 
\begin{align*}
&\int_0^T \intt \left(-w^\e \cdot \pa_t \psi + ((\theta_\e \star w^{\e}) \cdot \nabla_x) w^\e \cdot \psi
+ \tau(Dw^\e): D \psi + ( u_c^\e \cdot \nabla_x ) w^\e \cdot  \psi \right)\,dxds \\
&\quad =  \int_0^T \intt \bigg[ -\theta_\e \star \lt(\int_{\R^3} (\theta_\e \star w^\e +u_c^\e-v)f^\e\,dv\rt) + \inttr(\theta_\e\star w^\e+u_c^\e-v)f^\e\,dxdv\bigg]\cdot \psi\,dxds  \cr
&\qquad +  \intt u_0 \cdot \psi(0,\cdot)\,dx.
\end{align*} 
holds for all divergence-free $\psi \in C_c^\infty([0,T)\times\T^3)$, with regularized initial data
\bq\label{reg_init_e}
\begin{aligned}
&f^\e(0,x,v)= f_0^\e(x,v) := f_0(x,v)\mathds{1}_\e (v), \quad (x,v) \in \T^3 \times \R^3 \quad \mbox{and} \\
&w^\e(0,x)= (\theta_\e \star w_0) = (\theta_\e\star u_0) - \intt u_0\,dx,  \quad x \in \T^3,
\end{aligned}
\eq
and where $u_c^\e$ is defined as
\[\begin{aligned}
u_c^\e(t)&:= \lt(\intt u_0\,dx\rt)e^{-\int_0^t \intt \rho_{f^\e}\,dxds }  - \int_0^t e^{-\int_s^t \intt \rho_{f^\e}\,dxdr }\intt (\rho_{f^\e} (\theta_\e\star  w^\e) - m_{f^\e})\,dxds. \quad 
\end{aligned}\]

\paragraph{\it Step 1: Compactness and convergence of approximate solutions}

From Lemma \ref{unif_eps_est}, we obtain uniform-in-$n$ bounds for the sequence of solutions $(f^{n,\e}, w^{n,\e})$, which allows us to extract a subsequence (not relabeled) such that
\[
\begin{aligned}
&f^{n,\e} \stackrel{*}{\rightharpoonup} f^\e \quad \text{in } L^\infty(0,T;L^r(\T^3 \times \R^3)) \quad \text{for all } r \in [1,\infty],\\
&\rho_{f^{n,\e}} := \int_{\R^3} f^{n,\e}\,dv \stackrel{*}{\rightharpoonup} \rho_{f^\e} \quad \text{in } L^\infty(0,T;L^q(\T^3)) \quad \text{for all } q \in (1,5/3),\\
&m_{f^{n,\e}} := \int_{\R^3} v f^{n,\e}\,dv \stackrel{*}{\rightharpoonup} m_{f^\e} \quad \text{in } L^\infty(0,T;L^q(\T^3)) \quad \text{for all } q \in (1,5/4),\\
&w^{n,\e} \stackrel{*}{\rightharpoonup} w^\e \quad \text{in } L^\infty(0,T;L^2(\T^3)) \cap L^p(0,T;V^{1,p}).
\end{aligned}
\]

These weak-* convergences follow from the uniform bounds and Proposition \ref{PA.1}, which yields
\[
\begin{aligned}
&\sup_{0 \le t \le T} \|\rho_{f^{n,\e}}(t)\|_{L^q(\T^3)} \le C, \quad \forall q \in [1,5/3), \\
&\sup_{0 \le t \le T} \|m_{f^{n,\e}}(t)\|_{L^q(\T^3)} \le C, \quad \forall q \in [1,5/4),
\end{aligned}
\]
with $C$ independent of $n$.

Furthermore, the uniform moment bounds from Lemma \ref{unif_eps_est} and the velocity averaging result (Lemma \ref{vel_av} with $\beta = 2$) yield strong convergence of the macroscopic quantities:
\[\begin{aligned}
&\rho_{f^{n,\e}} \to \rho_{f^\e} \quad \text{in } L^{q_1}((0,T)\times\T^3), \quad \forall q_1 \in \lt(1,\frac53\rt),\\ 
&m_{f^{n,\e}} \to m_{f^\e} \quad \text{in } L^{q_2}((0,T)\times\T^3), \quad \forall q_2 \in \lt(1,\frac54\rt).
\end{aligned}\]

To obtain strong convergence of $w^{n,\e}$, we invoke the Aubin--Lions lemma. Specifically, we show the bound for $\pa_t w^{n,\e}$ in $\mathcal{Y}^*$ which is uniform both in $n$ and $\e$ and recall that $\mathcal{Y} := L^p(0,T; V^{1,p}) \cap L^{s'}((0,T)\times \T^3)$ with $s = \frac{5p}{8}$.  Note that this bound will also be used for the convergence of the stress tensor. For this, we first verify that $w^{n,\e}$ belongs to $L^{p'}(0,T;L^{\frac{5}{2} + \gamma}(\T^3))$ for some small $\gamma > 0$. Indeed, by interpolation, we have
\[
\|w^{n,\e}\|_{L^{\frac{5}{2}+\gamma}(\T^3)} \le \|\nabla_x w^{n,\e}\|_{L^p(\T^3)}^{\alpha} \|w^{n,\e}\|_{L^2(\T^3)}^{1-\alpha},
\]
with
\[
\alpha := \frac{\frac{1 + 2\gamma}{10 + 4\gamma}}{\frac{5}{6} - \frac{1}{p}},
\]
and it is shown that when $p > \frac{8}{5}$, one can choose $\gamma > 0$ such that
\bq\label{p_cond_coup}
\frac{\frac{1 + 2\gamma}{10 + 4\gamma}}{\frac{5}{6} - \frac{1}{p}} \cdot \frac{p}{p - 1} < p,
\eq
ensuring that $w^{n,\e} \in L^{p'}(0,T;L^{\frac{5}{2} + \gamma}(\T^3))$. Note that $\frac{\frac{1}{10}}{\frac 56 -\frac1p} \cdot \frac{p}{p-1} < p$ if and only if $p>\frac{29+\sqrt{91}}{25} =1.541...$ or $p< \frac{29 -\sqrt{91}}{25}$.

This integrability, together with the uniform both in $n$ and $\e$ bounds $w^{n,\e} \in L^{\frac{5p}{3}}((0,T)\times\T^3)$ and $(\theta_\e \star w^{n,\e})\cdot \nabla_x w^{n,\e} \in L^s((0,T)\times \T^3)$ shown in \eqref{w_est_sp}, is then used to obtain uniform-in-$n$ bounds for $\partial_t w^{n,\e}$ in the dual space $\mathcal{Y}^*$ by testing the equation against smooth functions $\phi \in \mathcal{Y}=L^p(0,T;V^{1,p})\cap L^{s'}((0,T)\times\T^3)$. Each nonlinear term in the fluid equation is estimated by H\"older's inequality and known bounds on $f^{n,\e}$, $w^{n,\e}$, and $u_c^{n,\e}$.  Indeed, we estimate
\begin{align*}
&\lt|\int_0^T \intt \pa_t w^{n,\e} \cdot \phi \,dxdt \rt| \\
&\quad \le \| (\theta_\e \star w^{n,\e}) \cdot \nabla_x w^{n,\e}\|_{L^s((0,T)\times \T^3)}\|\phi\|_{L^{s'}((0,T)\times \T^3)}   + C\int_0^T \|\nabla_x \phi(t)\|_{L^p(\T^3)}\|(1 + |Dw^{n,\e}|)\|_{L^p(\T^3)}^{p-1}\,dt \\
&\qquad +  \|(u_c^{n,\e} \cdot \nabla_x) w^{n,\e}\|_{L^{s}((0,T)\times \T^3)}\| \phi\|_{L^{s'}((0,T)\times\T^3))}   + \lt\| [\rho_{f^{n,\e}} (\theta_\e \star  w^{n,\e}+u_c^{n,\e})-m_{f^{n,\e}}] \cdot \phi \rt\|_{L^1((0,T)\times\T^3)} \\
&\qquad+ \|\rho_{f^{n,\e}} (\theta_\e \star  w^{n,\e}+u_c^{n,\e})-m_{f^{n,\e}}]\|_{L^{p'}(0,T;L^1(\T^3))}\|\phi\|_{L^p(0,T;L^1(\T^3))}  \\
&\quad \le C(1+\|w^{n,\e}\|_{L^\infty(0,T;L^2 (\T^3))}^2 +\|\nabla_x w^{n,\e}\|_{L^p((0,T)\times\T^3)}^2)\|\phi\|_{L^s((0,T)\times \T^3)}\\
&\qquad + C\|\phi\|_{L^p(0,T;V^{1,p})}( 1+ \|\nabla_x w^{n,\e}\|_{L^p((0,T)\times \T^3)}^{p-1})\\
&\qquad + C\Big(\|\rho_{f^{n,\e}} (\theta_\e \star w^{n,\e})\|_{L^{p'}(0,T;L^1(\T^3))} + \|\rho_{f^{n,\e}}\|_{L^{p'}(0,T;L^1(\T^3))} + \|m_{f^{n,\e}}\|_{L^{p'}(0,T;L^1(\T^3))}\Big)\|\phi\|_{L^p(0,T; V^{1,p})}\\
&\qquad + \lt\| [\rho_{f^{n,\e}} (\theta_\e \star  w^{n,\e}+u_c^{n,\e})-m_{f^{n,\e}}] \cdot \phi \rt\|_{L^1((0,T)\times\T^3)}\\
&\quad \le C\lt(1 + \|w^{n,\e}\|_{L^\infty(0,T;L^2(\T^3))}^2 + \|\nabla_x w^{n,\e}\|_{L^p((0,T)\times\T^3)}^p\rt)\|\phi\|_{\mathcal{Y}}\\
&\qquad + C\lt(\|\rho_{f^{n,\e}}\|_{L^\infty(0,T;L^{\frac 53 - \delta}(\T^3))}(\|w^{n,\e}\|_{L^{p'}(0,T;L^{\frac 52 +\gamma}(\T^3))}+\|u_c^{n,\e}\|_{L^\infty(0,T)}) + \|m_{f^{n,\e}}\|_{L^\infty(0,T;L^1(\T^3))} \rt)\|\phi\|_{\mathcal{Y}}\\
&\qquad +\lt\| [\rho_{f^{n,\e}} (\theta_\e \star  w^{n,\e}+u_c^{n,\e})-m_{f^{n,\e}}] \cdot \phi \rt\|_{L^1((0,T)\times\T^3)},
\end{align*}
where we used 
$p>\frac 85 >\frac 32$ to have the existence of sufficiently small $\gamma>0$ and $\delta>0$ satisfying \eqref{p_cond_coup} and
\bq\label{e_del_choice}
 \frac{1}{\frac 53 - \delta} + \frac{1}{\frac 52 +\gamma}=1,
\eq
and also the following bound:
\[\begin{aligned}
|u_c^{n,\e}(t)|&\le |u_c^{n,\e}(0)| +  \int_0^t \intt \rho_{f^{n,\e}} |\theta_\e\star  w^{n,\e}| + | m_{f_n}|)\,dxds\\
&\le |u_c^{n,\e}(0)| + \int_0^t (\|\rho_{f^{n,\e}}(s)\|_{L^{\frac 53 -\delta}}\|w^{n,\e}(s)\|_{L^{\frac 52 +\gamma}} + \|m_{f^{n,\e}}(s)\|_{L^1})\,ds\\
&\le C.
\end{aligned}\]

It now remains to estimate the last term 
\[
\lt\| [\rho_{f^{n,\e}} (\theta_\e \star  w^{n,\e}+u_c^{n,\e})-m_{f^{n,\e}}] \cdot \phi \rt\|_{L^1((0,T)\times\T^3)},
\] 
and we split the proof of this bound into two cases:\\

\noindent $\bullet$ (Case A: $\frac 85 < p<2$) In this case, we aim to show 
\[
\rho_{f^{n,\e}} (\theta_\e \star  w^{n,\e}+u_c^{n,\e})-m_{f^{n,\e}} \in L^s((0,T)\times\T^3),
\]
where $s = \frac{5p}{8} < \frac 54$. Clearly, $\rho_{f^{n,\e}} u_c^{n,\e}$, $m_{f^{n,\e}} \in L^s((0,T)\times\T^3)$. Moreover,

\[\begin{aligned}
\|\rho_{f^{n,\e}} (\theta_\e \star  w^{n,\e})\|_{L^{\frac{5p}{8}}((0,T)\times\T^3)} &\le \|\rho_{f^{n,\e}}\|_{L^\infty(0,T;L^{\frac53 -\delta'}(\T^3))}\|w^{n,\e}\|_{L^{\frac{5p}{8}}(0,T;L^{\frac{5p}{8-3p} +\gamma'}(\T^3)) }\\
&\le \|\rho_{f^{n,\e}}\|_{L^\infty(0,T;L^{\frac53 -\delta'}(\T^3))}\|\nabla w^{n,\e}\|_{L^{\frac{5p}{8}\alpha}(0,T;L^p(\T^3)) }^\alpha \|w^{n,\e}\|_{L^\infty(0,T;L^2(\T^3))}^{1-\alpha}, 
\end{aligned}\]
where $\delta'>0$ and $\gamma'>0$ are sufficiently small constants satisfying
\[
\frac{1}{\frac53-\delta'} + \frac{1}{\frac{5p}{8-3p}+\gamma'} = \frac{8}{5p},
\]
and
\[
\alpha :=\frac{\frac{\frac{11p-16}{8-3p}+\gamma'}{\frac{10p}{8-3p}+2\gamma' }}{\frac56 -\frac1p} \in (0,1).
\]
Since $\frac{5p}{8}\cdot \alpha <p$ for small enough $\gamma'$,  this implies 
\[
\rho_{f^{n,\e}} (\theta_\e \star  w^{n,\e}) \in L^s((0,T)\times\T^3) 
\]
uniformly both in $n$ and $\e$, and thus
\[
\left| \int_0^T \int_{\T^3} \partial_t w^{n,\e} \cdot \phi\,dxdt \right| \le C \|\phi\|_{\mathcal{Y}},
\]
where $C$ is independent of $n$ and $\e$.\\

\noindent $\bullet$ (Case B: $p>2$) Note that $p>2$ gives
\[
L^p(0,T;V^{1,p}) \subseteq L^2(0,T;L^6(\T^3))
\]
Thus, we show 
\[
\rho_{f^{n,\e}} (\theta_\e \star  w^{n,\e}+u_c^{n,\e})-m_{f^{n,\e}} \in L^2(0,T;L^{\frac 65}(\T^3)),
\]
which gives
\[\begin{aligned}
&\lt\| [\rho_{f^{n,\e}} (\theta_\e \star  w^{n,\e}+u_c^{n,\e})-m_{f^{n,\e}}] \cdot \phi \rt\|_{L^1((0,T)\times\T^3)} \cr
 &\quad \le \|\rho_{f^{n,\e}} (\theta_\e \star  w^{n,\e}+u_c^{n,\e})-m_{f^{n,\e}}\|_{L^2(0,T;L^{\frac65 }(\T^3))}\|\phi\|_{L^2(0,T;L^6(\T^3))}\\
&\quad \le  \|\rho_{f^{n,\e}} (\theta_\e \star  w^{n,\e}+u_c^{n,\e})-m_{f^{n,\e}}\|_{L^2(0,T;L^{\frac65 }(\T^3))}\|\phi\|_{L^p(0,T;V^{1,p})}.
\end{aligned}\]
First, one readily obtains  $\rho_{f^{n,\e}} u_c^{n,\e}$, $m_{f^{n,\e}} \in L^2(0,T;L^{\frac65}(\T^3))$. Moreover,
\[\begin{aligned}
\|\rho_{f^{n,\e}} (\theta_\e \star w^{n,\e})\|_{L^2(0,T;L^{\frac65 }(\T^3))} \le \|\rho_{f^{n,\e}}\|_{L^\infty(0,T;L^{\frac32}(\T^3))}\|w^{n,\e}\|_{L^2(0,T;L^6(\T^3))},
\end{aligned}\]
which also yields
\[
\left| \int_0^T \int_{\T^3} \partial_t w^{n,\e} \cdot \phi\,dxdt \right| \le C \|\phi\|_{\mathcal{Y}}.
\]

In either case, we obtain
\[
\partial_t w^{n,\e} \text{ is bounded in } \mathcal{Y}^*,
\]
and hence
\[\label{wne_lrlq}
w^{n,\e} \to w^\e \quad \text{strongly in } L^p(0,T;L^2(\T^3)),
\]
which further implies
\begin{equation}\label{w_str_range}
w^{n,\e} \to w^\e \quad \text{strongly in } L^r(0,T;L^q(\T^3)),
\end{equation}
for all admissible pairs $(r,q)$ satisfying
\[
q \in \left(1, \frac{3p}{3-p}\right), \quad r \in \left(1, \frac{5p - 6}{3 - \frac{6}{q}}\right).
\]

In particular, we obtain almost everywhere convergence of $w^{n,\e} \to w^\e$, and the convolution $\theta_\e \star w^{n,\e}$ also converges strongly. Together with the strong convergence of $(\rho_{f^{n,\e}}, m_{f^{n,\e}}, u_c^{n,\e})$ and weak-* convergence of $f^{n,\e}$, this would allow us to pass to the limit in the weak formulation and identify $(f^\e, w^\e)$ as a weak solution to the regularized system \eqref{reg_eq_e_k}--\eqref{reg_init_e}. 

\medskip

\paragraph{\it Step 2: Identification of the limit system}

We now verify that the limiting pair $(f^\e, w^\e)$ indeed satisfies the regularized system \eqref{reg_eq_e_k}--\eqref{reg_init_e} in the sense of distributions. The weak convergence results established in the previous step ensure that all terms in the weak formulation converge appropriately, provided that the nonlinear terms admit suitable continuity or compactness.

Specifically, for any test functions $\phi \in \mathcal{C}_c^\infty([0,T)\times\T^3\times\R^3)$ and $\psi \in \mathcal{C}_c^\infty([0,T)\times\T^3)$ with $\nabla_x \cdot \psi = 0$, we need to show:
\begin{itemize}
\item[(i)] \quad $\displaystyle \int_0^T \iint_{\T^3 \times \R^3} (\theta_\e \star w^{n,\e}) \cdot \nabla_v \phi \, f^{n,\e}\,dvdxdt \to \int_0^T \iint_{\T^3 \times \R^3} (\theta_\e \star w^\e) \cdot \nabla_v \phi \, f^\e \,dvdxdt,$ \\[1mm]
\item[(ii)] \quad $\displaystyle \int_0^T \int_{\T^3} \theta_\e \star \left( \int_{\R^3} (\theta_\e \star w^{n,\e} + u_c^{n,\e} - v) f^{n,\e} \,dv \right) \cdot \psi \,dxds \\
 \mbox{ } \hspace{3cm} \to \int_0^T \int_{\T^3} \theta_\e \star \left( \int_{\R^3} (\theta_\e \star w^\e + u_c^\e - v) f^\e \,dv \right) \cdot \psi \,dxds$, \\[1mm]
\item[(iii)] \quad $\displaystyle \int_0^T \int_{\T^3} \tau(D w^{n,\e}) : D\psi \,dxds \to \int_0^T \int_{\T^3} \tau(D w^\e) : D\psi \,dxds$.
\end{itemize}

\medskip

\noindent (i) The convergence in the kinetic term follows from the strong convergence of $w^{n,\e} \to w^\e$ in $L^p(0,T;L^2(\T^3))$ and the weak-* convergence of $f^{n,\e}$. Indeed, we write
\[
\begin{aligned}
&\left|\int_0^T \iint_{\T^3 \times \R^3} \left[(\theta_\e \star w^{n,\e}) f^{n,\e} - (\theta_\e \star w^\e) f^\e \right] \cdot \nabla_v \phi \,dvdxdt \right| \\
&\quad \le \int_0^T \iint_{\T^3 \times \R^3} |\theta_\e \star (w^{n,\e} - w^\e)| \, f^{n,\e} \, |\nabla_v \phi|\,dvdxdt + \left|\int_0^T \iint_{\T^3 \times \R^3} (\theta_\e \star w^\e) \cdot \nabla_v \phi (f^{n,\e} - f^\e) \,dvdxdt \right|\cr
&\quad \le C\|\theta_\e \star(w^{n,\e} -w^\e)\|_{L^p(0,T;L^2(\T^3))}|\mbox{supp}_v \phi| \|\nabla_v \phi\|_{L^\infty((0,T)\times\T^3\times\R^3)}\|f^{n,\e}\|_{L^\infty((0,T)\times\T^3\times\R^3)}\\
&\qquad + \lt|\int_0^T \inttr \theta_\e \star w^\e \cdot \nabla_v \phi (f^{n,\e} -f^\e)\,dvdxdt \rt|.
\end{aligned}
\]
The first term vanishes due to strong convergence of $w^{n,\e}$ in $L^p(0,T;L^2(\T^3))$ and boundedness of $f^{n,\e}$ in $L^\infty$, while the second term vanishes due to the weak-* convergence of $f^{n,\e}$. Hence, the desired convergence holds.

\medskip

\noindent{(ii)} For the nonlinear drag term, we first observe that 
\begin{equation}\label{ucLinfty}u_c^{n,\e} \to u_c^\e \mbox{ strongly in }L^\infty(0,T),
\end{equation} 
which follows from the strong convergence of $\rho_{f^{n,\e}}$, $m_{f^{n,\e}}$, and $w^{n,\e}$, and continuity of the exponential and integral operators defining $u_c^{n,\e}$. Indeed, we deduce
\begin{align*}
&|u_c^{n,\e}(t) - u_c^\e(t)| \cr
&\quad \le C\|u_0\|_{L^1(\T^3)} \int_0^t \intt |\rho_{f^{n,\e}} - \rho_{f^\e}|\,dxds\\
&\qquad +  C\int_0^t \|\rho_{f^{n,\e}}(s)(\theta_\e \star w^{n,\e}(s)) - m_{f^{n,\e}}(s)\|_{L^1}  \int_s^t \intt |(\rho_{f^{n,\e}} - \rho_{f^\e})(r)|\,dxdrds\\
&\qquad + C\int_0^t \intt|(\rho_{f^{n,\e}} - \rho_{f^\e})(s)| |\theta_\e  \star w^{n,\e}(s)|\,dxds\\
&\qquad + C\int_0^t \intt \lt(\rho_{f^\e} (s) |\theta_\e \star (w^{n,\e} - w^\e)(s)| + |(m_{f^{n,\e}} - m_{f^\e})(s)| \rt)\,dxds\\
&\quad \le C \|u_0\|_{L^1(\T^3)}\|\rho_{f^{n,\e}} - \rho_{f^\e}\|_{L^{\frac 54}((0,T)\times\T^3)}\\
&\qquad + C\lt(\|\rho_{f^{n,\e}}\|_{L^\infty(0,T;L^{\frac 53 -\delta}(\T^3)}\|w^{n,\e}\|_{L^{p'}(0,T;L^{\frac 52 +\gamma}(\T^3))} + \|m_{f^{n,\e}}\|_{L^\infty(0,T;L^1(\T^3)}\rt)\|\rho_{f^{n,\e}} - \rho_{f^\e}\|_{L^{\frac 54}((0,T)\times\T^3)}\\
&\qquad+ C\|\rho_{f^{n,\e}} - \rho_{f^\e}\|_{L^{\frac 53-\delta}((0,T)\times\T^3)} \|w^{n,\e}\|_{L^{\frac 52 +\gamma}((0,T)\times\T^3)}\\
&\qquad + C\|\rho_{f^\e} \|_{L^\infty(0,T;L^{\frac 53-\delta}(\T^3))} \|w^{n,\e} - w^\e\|_{L^2(0,T;L^{\frac 52+\gamma}(\T^3))} + \|m_{f^{n,\e}} - m_{f^\e}\|_{L^{\frac 65}((0,T)\times\T^3)},
\end{align*}
where we used aforementioned strong convergences of $(\rho_{f^{n,\e}}, m_{f^{n,\e}}, w^{n,\e})$, and chose sufficiently small $\gamma>0$ and $\delta>0$ satisfying
\bq\label{e_del_choice2}
\frac{\frac{1 + 2\gamma}{10 + 4\gamma}}{\frac{5}{6} - \frac{1}{p}} \cdot \lt(\frac 52+\gamma\rt) < p,\quad \frac{1}{\frac 53 - \delta} + \frac{1}{\frac 52 +\gamma}=1.
\eq
Note that  $p>\frac85 > \frac32$ gives $\frac{\frac{1}{10}}{\frac56 - \frac1p} \cdot \frac52 <p$, which guarantees the existence of $\gamma$ satisfying the first relation in \eqref{e_del_choice2}. This in turn implies $w^{n,\e} \in L^{\frac52+\gamma}((0,T)\times\T^3)$ uniformly in $n$ and $\e$. Moreover, $w^{n,\e}$ converges to $w^\e$ strongly in $L^2(0,T;L^{\frac 52+\gamma}(\T^3))$ according to \eqref{w_str_range}. Hence, we obtain
\begin{align*}
\Bigg|\int_0^T&\intt \theta_\e \star\lt[\intr (\theta_\e \star w^{n,\e} + u_c^{n,\e}-v) f^{n,\e}\,dv - \intr (\theta_\e \star w^\e + u_c^\e-v) f^\e\,dv\rt] \cdot \psi \,dxds\Bigg|\\
&\le \int_0^T \intt \lt| \lt[(\rho_{f^{n,\e}} (\theta_\e \star w^{n,\e} + u_c^{n,\e}) -m_{f^{n,\e}}) -   (\rho_{f^\e} ( \theta_\e \star w^\e + u_c^\e) -m_{f^\e})\rt]\cdot (\theta_\e \star\psi)\rt|\,dxds\\
&\le \int_0^T \intt \lt|(\rho_{f^{n,\e}} - \rho_{f^\e}) (\theta_\e \star w^{n,\e} + u_c^{n,\e}) \cdot \theta_\e\star \psi\rt|\,dxds\\
&\quad + \int_0^T \intt \lt| \rho_{f^\e} (\theta_\e\star(w^{n,\e} - w^\e) + (u_c^{n,\e}- u_c^\e) )\cdot (\theta_\e\star\psi) \rt|\,dxds + \int_0^t \intt \lt|(m_{f^{n,\e}} - m_{f^\e} )\cdot \theta_\e\star\psi  \rt|\,dxds\\
&\le \|\rho_{f^{n,\e}} - \rho_{f^\e}\|_{L^{\frac 53-\delta}((0,T)\times\T^3)}(\|w^{n,\e}\|_{L^{\frac52 +\gamma}((0,T)\times\T^3)} + \|u_c^{n,\e}\|_{L^\infty(0,T)})\|\theta_\e\star \psi\|_{L^\infty((0,T)\times\T^3)}\\
&\quad+ C\|\rho_{f^\e}\|_{L^\infty(0,T;L^{\frac 53- \delta}(\T^3))}\|w^{n,\e} - w^\e\|_{L^2(0,T;L^{\frac52 + \gamma}(\T^3))}\|\theta_\e\star\psi\|_{L^\infty((0,T)\times\T^3)}\\
&\quad +C\|\rho_{f^\e}\|_{L^\infty(0,T;L^1(\T^3))}\|u_c^{n,\e} - u_c^\e\|_{L^\infty(0,T)}\|\theta_\e\star\psi\|_{L^\infty((0,T)\times\T^3)}\\
&\quad + C\|m_{f^{n,\e}} -m_{f^\e}\|_{L^{\frac 65}((0,T)\times \T^3)}\|\theta_\e\star\psi\|_{L^\infty((0,T)\times\T^3)}\\
&\longrightarrow 0 \quad \mbox{as} \quad n \to \infty .
\end{align*}


\noindent{(iii)}
 Finally, to handle the nonlinear stress tensor, we observe that by Lemma~\ref{unif_eps_est} 
 $\tau(Dw^{n,\e})$ is uniformly bounded in $L^{p'}((0,T) \times \T^3)$, hence
 \begin{equation}\label{tdun_weak}
 \tau(Dw^{n,\e}) {\rightharpoonup}  \chi \quad \mbox{ weakly in }L^{p'}((0,T) \times \T^3).
 \end{equation}
Since our nonlinear viscous term is strictly monotone, we can apply similar arguments as in  \cite{FMS97, FMS00} to obtain 
    \begin{equation}\label{Dwne_ae}	
    Dw^{n,\e} \to Dw^\e \quad \mbox{a.e. in }(0,T) \times \T^3. 
    \end{equation}
We provide a detailed proof of the above in Appendix \ref{app.A} for the readers' convenience. Let us note that at this point we still keep the regularization in the convective term, and thus the arguments may not be entirely optimal. Combining \eqref{tdun_weak} and \eqref{Dwne_ae}, together with continuity of $\tau^\e (\cdot)$, we obtain
\[
\int_0^T \int_{\T^3} \tau(D w^{n,\e}) : D\psi \,dxds \to \int_0^T \int_{\T^3} \tau(D w^\e) : D\psi \,dxds.
\]

We conclude that the pair $(f^\e, w^\e)$ is a weak solution to the regularized system \eqref{reg_eq_e_k}--\eqref{reg_init_e}.

\medskip

\paragraph{\it Step 3: Energy inequality for the limit system}

It remains to verify that the limiting pair $(f^\e, w^\e)$ satisfies the energy inequality corresponding to the regularized system \eqref{reg_eq_e_k}--\eqref{reg_init_e}.

From Lemma \ref{unif_eps_est}, we already have the uniform-in-$n$ energy estimate for the approximate solution $(f^{n,\e}, w^{n,\e})$ as in \eqref{unif_est_soln}. We now pass to the limit $n \to \infty$ in this inequality. The weak-* convergence of $f^{n,\e}$ in $L^\infty(0,T;L^1(\T^3 \times \R^3))$ yields
\[
\liminf_{n \to \infty} \iint_{\T^3 \times \R^3} \frac{|v|^2}{2} f^{n,\e}(t)\,dvdx \ge \iint_{\T^3 \times \R^3} \frac{|v|^2}{2} f^\e(t)\,dvdx.
\]
 Similarly, since $\{ w^{n,\e} \}_n$ is uniformly bounded in $L^\infty(0,T;L^2(\T^3))$, and  $\{ D w^{n,\e} \}_n$ is uniformly bounded in $L^p((0,T)\times \T^3)$,
by weak convergence and convexity of $|\cdot|^2$ and $|\cdot|^p$, we obtain
\[
\begin{aligned}
&\liminf_{n \to \infty} \int_{\T^3} \frac{|w^{n,\e}(t)|^2}{2}\,dx \ge \int_{\T^3} \frac{|w^\e(t)|^2}{2}\,dx, \\
&\liminf_{n \to \infty} \int_0^T \int_{\T^3} \mu |D w^{n,\e}|^p\,dxdt \ge \int_0^T \int_{\T^3} \mu |D w^\e|^p\,dxdt.
\end{aligned}
\]

The convergence \eqref{ucLinfty} implies
\[
\lim_{n \to \infty} |u_c^{n,\e}(t)|^2 = |u_c^\e(t)|^2,
\]
and we also deduce
\[
\liminf_{n \to \infty} \int_0^T \iint_{\T^3 \times \R^3} |\theta_\e \star w^{n,\e} + u_c^{n,\e} - v|^2 f^{n,\e}\,dvdxdt \ge \int_0^T \iint_{\T^3 \times \R^3} |\theta_\e \star w^\e + u_c^\e - v|^2 f^\e\,dvdxdt,
\]
by Fatou's lemma and the pointwise convergence of the integrand.

For the initial data, we have strong convergence:
\[
\begin{aligned}
&f_0^{n,\e} \to f_0^\e \quad \text{strongly in } L^1 \cap L^\infty(\T^3 \times \R^3),\\
&w_0^{n,\e} \to w_0^\e \quad \text{strongly in } L^2(\T^3),
\end{aligned}
\]
which implies convergence of the right-hand side of the inequality. Therefore, the limit $(f^\e, w^\e)$ satisfies the energy inequality:
\begin{align}\label{en_eq_e}
\begin{aligned}
&\sup_{t \in (0,T)} \left( \iint_{\T^3 \times \R^3} \frac{|v|^2}{2} f^\e\,dvdx + \int_{\T^3} \frac{|w^\e|^2}{2}\,dx + \frac{1}{2} |u_c^\e|^2 \right) \\
&\quad + \int_0^T \int_{\T^3} \mu|D w^\e|^p\,dxdt + \int_0^T \iint_{\T^3 \times \R^3} |\theta_\e \star w^\e + u_c^\e - v|^2 f^\e\,dvdxdt  \\
&\qquad \le \iint_{\T^3 \times \R^3} \frac{|v|^2}{2} f_0^\e\,dvdx + \int_{\T^3} \frac{|w_0^\e|^2}{2}\,dx + \frac{1}{2} |u_c^\e(0)|^2.  
\end{aligned}
\end{align}
This completes the proof that the limit $(f^\e, w^\e)$ is a weak solution to the regularized system and satisfies the associated energy inequality.

\subsubsection{The limit $\e \to 0$: passing to the original system}

We now pass to the limit $\e \to 0$ to construct a global-in-time weak solution to the original system \eqref{main_eq} in the sense of Definition \ref{def_weak}. The starting point is the approximate solution $(f^\e, w^\e)$ to the regularized system \eqref{reg_eq_e_k}--\eqref{reg_init_e}, obtained in the previous subsection.

\medskip

\paragraph{\it Step 1: Compactness and convergence of approximate solutions}

From the uniform energy bounds and a priori estimates, we extract a subsequence (still indexed by $\e$) such that
\[
\begin{aligned}
&f^\e \stackrel{*}{\rightharpoonup} f \quad \text{in } L^\infty(0,T;L^r(\T^3 \times \R^3)) \quad \forall \, r \in [1,\infty],\\
&\rho_{f^\e} \stackrel{*}{\rightharpoonup} \rho_f \quad \text{in } L^\infty(0,T;L^q(\T^3)), \quad \forall \, q \in (1,5/3),\\
&m_{f^\e} \stackrel{*}{\rightharpoonup} m_f \quad \text{in } L^\infty(0,T;L^q(\T^3)), \quad \forall \, q \in (1,5/4),\\
&w^\e \to w \quad \text{a.e. and in } L^r(0,T;L^q(\T^3)) \quad \forall \, q \in \left(1, \frac{3p}{3-p} \right), \ r \in \left(1, \frac{5p-6}{3 - \tfrac{6}{q}} \right),\\
&\nabla_x w^\e \rightharpoonup \nabla_x w \quad \text{in } L^p((0,T)\times \T^3).
\end{aligned}
\]
By velocity averaging and moment estimates, we also obtain strong convergence:
\[
\begin{aligned}
&\rho_{f^\e} \to \rho_f \quad \text{in } L^{q_1}((0,T)\times \T^3), \quad \forall \,q_1 \in \lt(1, \frac53 \rt),\\
&m_{f^\e} \to m_f \quad \text{in } L^{q_2}((0,T)\times \T^3), \quad \forall \, q_2 \in \lt(1, \frac54\rt).
\end{aligned}
\]


\paragraph{\it Step 2: Identification of the limiting system}

We now verify that the limit functions $(f, w)$ indeed satisfy the original system \eqref{main_eq} in the sense of Definition \ref{def_weak}. This follows by passing to the limit in the weak formulations satisfied by the approximate solutions $(f^\e, w^\e)$. For any test functions $\phi\in\mathcal{C}_c^\infty([0,T)\times\T^3\times\R^3)$ with $\phi(T,\cdot, \cdot)=0$ and $\psi \in \mathcal{C}([0,T]\times\T^3)$ with $\nabla_x  \cdot \psi = 0$, we show 
\begin{itemize}
\item[(i)] \quad $\displaystyle \int_0^T \inttr (\theta_\e \star w^\e \cdot \nabla_v \phi) f^\e\,dvdxdt \to \int_0^T \inttr w \cdot \nabla_v \phi f\,dvdxdt$,\\
\item[(ii)] \quad $\displaystyle \int_0^T\intt \theta_\e \star\lt(\intr (\theta_\e \star w^\e + u_c^\e-v) f^\e\,dv\rt) \cdot \psi \,dxds \to \int_0^T\inttr (w+u_c-v)\cdot \psi f\,dvdxds$,\\
\item[(iii)] \quad $\displaystyle \int_0^T \intt \tau(Dw^\e) : D\psi\,dxds \to \int_0^T \intt \tau(Dw) : D\psi\,dxds $.
\end{itemize}

\medskip

\noindent (i) We estimate
\begin{align*}
&\bigg|\int_0^T \inttr ((\theta_\e \star w^\e)  f^\e - w f)\cdot \nabla_v \phi \,dvdxdt \bigg|\\
&\quad \le\int_0^T \inttr  |\theta_\e \star (w^\e - w)| f^\e |\nabla_v \phi|\,dvdxdt + \int_0^T \inttr |(\theta_\e \star w) - w| |\nabla_v \phi| f^{n,\e}\,dvdxdt\\
&\qquad + \lt|\int_0^T \inttr w\cdot \nabla_v \phi (f^\e -f)\,dvdxdt \rt|\\
&\quad \le C\|\theta_\e \star(w^\e -w)\|_{L^{p'}(0,T;L^2(\T^3))}|\mbox{supp}_v \phi| \|\nabla_v \phi\|_{L^\infty((0,T)\times\T^3\times\R^3)}\|f^\e\|_{L^\infty((0,T)\times\T^3\times\R^3)}\\
&\qquad + C(1+\|w\|_{L^\infty(0,T;L^2(\T^3))}+\|\nabla_x w\|_{L^p((0,T)\times\T^3)})|\mbox{supp}_v \phi| \|\nabla_v \phi\|_{L^\infty((0,T)\times\T^3\times\R^3)}\\
&\qquad \quad \times\|f^\e\|_{L^\infty((0,T)\times\T^3\times\R^3)} \intt |y|^\beta \theta_\e (y)\,dy\\
&\qquad + \lt|\int_0^T \inttr w\cdot \nabla_v \phi (f^\e -f)\,dvdxdt \rt| \to 0 \quad \mbox{as} \quad \e \to 0,
\end{align*}
where we used
\bq\label{quo_est}
\begin{aligned}
\intt& ((\theta_\e \star w) - w) h(x)\,dx \\
&= \iint_{\T^3\times\T^3} |y|^\beta\theta_\e(y) \lt(\frac{w(x-y)-w(x)}{|y|^\beta} \rt) h(x)\,dydx \\
&\le \intt ||y|\theta_\e(y)| \|h|w(\cdot-y)-w(\cdot)|^{1-\beta}\|_{L^{\frac{p}{p-\beta}}(\T^3)}\lt(\intt  \lt|\frac{w(x-y)-w(x)}{|y|} \rt|^p\,dx\rt)^{\frac\beta p}\,dy\\
&\le C\|h\|_{L^q(\T^3)}\|w(t)\|_{L^{\frac{1-\beta}{1-\frac \beta p - \frac 1q}}(\T^3)}^{1-\beta}\|\nabla_x w(t)\|_{L^p(\T^3)}^\beta \intt |y|^\beta\theta_\e(y)\,dy,
\end{aligned}
\eq
and we chose $q=\infty$ and $\beta \in (0,1)$ be sufficiently small so that $\frac{p}{p-\beta} \le 2$.

\medskip

\noindent (ii) We similarly the strong convergence of $u_c^\e$ to $u_c$, i.e. 
\[
\begin{aligned}
&|u_c^\e(t) - u_c(t)| \cr
&\quad \le  C \|u_0\|_{L^1(\T^3)}\|\rho_{f^\e} - \rho_f\|_{L^{\frac 54}((0,T)\times\T^3)}\\
&\qquad + C\lt(\|\rho_{f^\e}\|_{L^\infty(0,T;L^{\frac 53 -\delta}(\T^3)}\|w^\e\|_{L^{p'}(0,T;L^{\frac 52 +\gamma}(\T^3))} + \|m_{f^\e}\|_{L^\infty(0,T;L^1(\T^3)}\rt)\|\rho_{f^\e} - \rho_f\|_{L^{\frac 54}((0,T)\times\T^3)}\\
&\qquad+ C\|\rho_{f^\e} - \rho_f\|_{L^{\frac 53-\delta}((0,T)\times\T^3)} \|w^\e\|_{L^{\frac 52 +\gamma}((0,T)\times\T^3)}\\
&\qquad + C\|\rho_f \|_{L^\infty(0,T;L^{\frac 53-\delta}(\T^3))} \|w^\e - w\|_{L^2(0,T;L^{\frac 52+\gamma}(\T^3))} + \|m_{f^\e} - m_f\|_{L^{\frac 65}((0,T)\times\T^3)},
\end{aligned}
\]
where we also used the same $\e>0$ and $\delta>0$ as \eqref{e_del_choice}  so that $w^\e$ converges to $w$ strongly in $L^2(0,T;L^{\frac 52+\gamma}(\T^3))$, and $w^\e$ is uniformly bounded in $L^{\frac 52 + \gamma}((0,T)\times\T^3)$. Then we obtain
\begin{align*}
&\Bigg|\int_0^T\intt \theta_\e \star\lt[\intr (\theta_\e \star w^\e + u_c^\e-v) f^\e\,dv\rt] \cdot \psi \,dxds -\int_0^t\inttr (w+u_c-v)\cdot \psi f\,dvdxds\Bigg|\\
&\quad \le \int_0^T \intt \lt|(\rho_{f^\e} (\theta_\e \star w^\e + u_c^\e) -m_{f^\e})\cdot (\theta_\e \star \psi) -   (\rho_f ( w + u_c) -m_f)\cdot \psi\rt|\,dxds\\
&\quad \le \int_0^T \intt \lt|(\rho_{f^\e} - \rho_f) (\theta_\e \star w^\e + u_c^\e) \cdot \theta_\e\star \psi\rt|\,dxds\\
&\qquad + \int_0^T \intt \lt| \rho_f (\theta_\e\star(w^\e - w) +((\theta_\e\star w) - w)+ (u_c^\e- u_c) )\cdot (\theta_\e\star\psi) \rt|\,dxds\\
&\qquad + \int_0^T \intt \lt|(m_{f^\e} - m_f )\cdot \theta_\e\star\psi +(\rho_f ( w +u_c) - m_f)\cdot ((\theta_\e\star\psi)-\psi)  \rt|\,dxds\\
&\quad \le \|\rho_{f^\e} - \rho_f\|_{L^{\frac 53-\delta}((0,T)\times\T^3)}(\|w^\e\|_{L^{\frac52 +\gamma}((0,T)\times\T^3)} + \|u_c^\e\|_{L^\infty(0,T)})\|\theta_\e\star \psi\|_{L^\infty((0,T)\times\T^3)}\\
&\qquad+ C\|\rho_f\|_{L^\infty(0,T;L^{\frac 53- \delta}(\T^3))}\|w^\e - w\|_{L^2(0,T;L^{\frac52 + \gamma}(\T^3))}\|\theta_\e\star\psi\|_{L^\infty((0,T)\times\T^3)}\\
&\qquad + C \|\rho_f\|_{L^\infty(0,T;L^{\frac 53-\delta'}(\T^3))}\|w\|_{L^{p'}(0,T;L^{\frac 52 +\gamma'}(\T^3))}^{1-\beta}\|\nabla_x w\|_{L^p((0,T)\times \T^3)}^\beta\|\theta_\e\star\psi\|_{L^\infty((0,T)\times\T^3)}\intt |y|^\beta \theta_\e(y)\,dy\\
&\qquad +C\|\rho_f\|_{L^\infty(0,T;L^1(\T^3))}\|u_c^\e - u_c\|_{L^\infty(0,T)}\|\theta_\e\star\psi\|_{L^\infty((0,T)\times\T^3)}\\
&\qquad + C\|m_{f^\e} -m_f\|_{L^{\frac 65}((0,T)\times \T^3)}\|\theta_\e\star\psi\|_{L^\infty((0,T)\times\T^3)}\\
&\qquad + C\lt(\|\rho_f\|_{L^\infty(0,T;L^{\frac 53-\delta}(\T^3))}\lt(\|w\|_{L^{p'}(0,T;L^{\frac 52+\gamma}(\T^3))} +\|u_c\|_{L^\infty}\rt) + \|m_f\|_{L^\infty(0,T;L^1(\T^3))}\rt) \cr
&\qquad \quad \times \|\nabla_x \psi\|_{L^\infty((0,T)\times\T^3)}\intt |y|\theta_\e(y)\,dy\\
&\quad \longrightarrow 0 \quad \mbox{as} \quad \e \to 0,
\end{align*}
where we used the same $\gamma$ and $\delta$ as \eqref{e_del_choice} as before and also the estimate \eqref{quo_est} with
\[
q= \frac 53 - \delta', \quad \frac{1-\beta}{1-\frac \beta p -\frac 1q} = \frac 52 + \gamma',
\]
where $\gamma'$, $\delta'$ and $\beta\in(0,1)$ are chosen sufficiently small to ensure $w \in L^{p'}(0,T;L^{\frac 52+\gamma'}(\T^3))$.

\medskip

\noindent (iii)  For the stress term, we apply arguments similar to those used in the passage $n \to \infty$. By \eqref{en_eq_e}, the sequence
 $\tau(Dw^{\e})$ is uniformly bounded in $L^{p'}((0,T) \times \T^3)$, so that (up to subsequences) 
 \[
 \tau(Dw^{\e}) {\rightharpoonup}  \chi  \quad \mbox{weakly in }  L^{p'}((0,T) \times \T^3). 
 \]
 Thanks to the strict monotonicity of the nonlinear stress tensor, we again use the bounded truncation method as before to deduce
    \begin{equation}\label{Dwe_ae}
 	Dw^{\e} \to Dw \quad \mbox{a.e. in }(0,T) \times \T^3.
    \end{equation}
By the continuity of $\tau (\cdot)$, it follows that
\[
\tau(Dw^\e) \rightharpoonup \tau(Dw) \quad \text{weakly in }L^{p'}((0,T) \times \T^3).
\]
We note that the condition $p > \frac{8}{5}$ is essential at this stage to guarantee the applicability of the bounded truncation method leading to \eqref{Dwe_ae}.  
 
Hence, $(f, w)$ solves \eqref{main_eq} in the sense of Definition \ref{def_weak}. Setting $u := w + u_c$ completes the construction of the weak solution.

\medskip

\paragraph{\it Step 3: Energy inequality}

Finally, we pass to the limit in the energy inequality. The uniform bounds from the approximate system and the lower semicontinuity of convex functionals imply that the limit $(f, u)$ satisfies
\[
\begin{aligned}
&\sup_{t \in (0,T)}\left( \inttr \frac{|v|^2}{2} f\,dvdx + \intt \frac{|u|^2}{2}\,dx \right) +  \mu\int_0^T \intt   |D u|^p\,dxdt  + \int_0^T \inttr |u-v|^2 f\,dvdxdt \\
&\quad \le \inttr \frac{|v|^2}{2} f_0\,dvdx + \intt \frac{|u_0|^2}{2}\,dx.
\end{aligned}
\]

This concludes the proof of Theorem \ref{thm_main}, establishing the global existence of weak solutions to the coupled Vlasov--non-Newtonian fluid system.

%
%
%
%
%
\section{A priori estimate for the large-time behaviors}\label{sec_lt}
In this section, we establish decay estimates for global-in-time weak solutions to the system \eqref{main_eq} by deriving a differential inequality for a suitably defined modulated energy functional.

We assume, without loss of generality, that the total mass of particles is normalized to unity:
\[
\inttr f\,dxdv = \inttr f_0\,dxdv = 1, \quad \forall\, t \geq 0.
\]

%
%
%
%
\subsection{Modulated energy estimate}

To capture the large-time behavior of both the kinetic and fluid components, we consider a modulated energy functional $\calE(t)$ and a corresponding dissipation $\calD(t)$ given by
\[
\calE := \frac12\inttr |v - v_c|^2 f\,dxdv + \frac12\intt |u - u_c|^2 + \frac14|u_c - v_c|^2
\]
and  
\[
\calD:= \mu\intt | D u|^p \,dx + \inttr |u - v|^2 f\,dxdv,
\]
where the averaged quantities $v_c(t)$ and $u_c(t)$ are given as
\[
v_c := \inttr vf\,dxdv \quad \mbox{and} \quad u_c:= \intt u\,dx.
\]
The modulated energy $\calE(t)$ measures the deviation of the kinetic and fluid velocities from their respective means and mutual alignment, while $\calD(t)$ encodes both internal dissipation in the fluid and drag-induced relaxation toward the kinetic mean.

We first establish a key differential identity:

\begin{lemma} Let $(f,u)$ be a solution to the system \eqref{main_eq} with sufficient regularity on the time interval $[0,T]$. Then we have
\[
\frac{d}{dt} \calE(t) + \calD(t) = 0.
\]
\end{lemma}
\begin{proof}Since the velocity means satisfy
\[
\inttr (v - v_c)f\,dxdv = 0  = \intt (u - u_c)\,dx.
\]
we obtain
\[
\frac12 \frac{d}{dt} \inttr |v - v_c|^2 f\,dxdv = \inttr (v - v_c) \cdot (u-v)f\,dxdv
\] 
and
\[
\frac12 \frac{d}{dt} \intt | u - u_c|^2 \,dx = -  \mu \intt |D u|^p\,dx - \inttr (u - u_c) \cdot (u-v)f\,dxdv.
\]
Summing the above estimates yields
\begin{align*}
&\frac12\frac{d}{dt}\lt( \inttr |v - v_c|^2 f\,dxdv +  \intt | u - u_c|^2 \,dx  \rt) + \mu \intt |D u|^p\,dx +  \inttr |u - v|^2 f\,dxdv \cr
&\quad  = \inttr (u_c - v_c) \cdot (u-v)f\,dxdv.
\end{align*}
On the other hand, by using the conservation of total momentum, we get
\[
\frac14 \frac{d}{dt} | v_c - u_c|^2 = - (v_c - u_c) \cdot \frac{d u_c}{dt} =  -\inttr (u_c - v_c) \cdot (u-v)f\,dxdv.
\]
Adding the two relations cancels the right-hand side and gives the desired identity.
\end{proof}

%
%
%
%
\subsection{Proof of Theorem \ref{thm_main2}} 
This balanced identity plays a central role in our large-time analysis. To extract decay estimates, we must estimate $\calE(t)$ from below in terms of $\calD(t)$. To this end, we expand the drag term:
\begin{align*}
&\inttr | u - v|^2 f\,dxdv \cr
&\quad = \inttr | u - u_c + u_c - v_c + v_c - v|^2 f\,dxdv\cr
&\quad = \inttr \lt( |u - u_c|^2 + | u_c - v_c|^2 + |v_c - v|^2 + 2(u - u_c) \cdot (u_c - v_c) + 2(u - u_c) \cdot (v_c - v) \rt)f\,dxdv.
\end{align*}
Thus, we obtain
\begin{align*}
|u - v_c|^2 + \inttr |v_c - v|^2 f\,dxdv &= \inttr |u-v|^2 f\,dxdv - \inttr |u - u_c|^2 f\,dxdv\cr
&\quad - 2(u_c - v_c) \cdot \inttr (u - u_c)f\,dxdv\cr
&\quad - \inttr (u - u_c)\cdot (v_c - v)f\,dxdv.
\end{align*}
Here, we estimate the last two terms on the right-hand side of the above as
\[
\lt|2(u_c - v_c) \cdot \inttr (u - u_c)f\,dxdv \rt| \leq \frac34 |u_c - v_c|^2 + \frac43 \inttr |u - u_c|^2 f\,dxdv
\]
and
\[
\lt|\inttr (u - u_c)\cdot (v_c - v)f\,dxdv \rt| \leq \frac12 \inttr |v_c - v|^2 f\,dxdv + 2\inttr |u-u_c|^2 f\,dxdv.
\]
This implies
\bq\label{ineq_E}
\calE(t) \leq \inttr |u-v|^2 f\,dxdv + 3\inttr |u-u_c|^2 f\,dxdv + \frac12 \intt |u - u_c|^2\,dx.
\eq
\subsubsection{The case $p > 2$} 
We first consider the case $p > 2$ and claim that
\[
\calE(t)^{\frac p2} \ls \calD(t), \quad \forall\, t \geq 0.
\]

For $p \in (2,3)$, the H\"older inequality and the Gagliardo--Nirenberg--Sobolev inequality yield
\[
\intt |u - u_c|^2 \rho_f\,dx \leq \|\rho_f\|_{L^{\frac{3p}{5p-6}}} \|u - u_c\|_{L^{\frac{1}{\frac1p - \frac13 }}}^2 \leq C\|\rho_f\|_{L^r}\|Du\|_{L^p}^2,
\]
where $\rho_f := \intr f\,dv$ denotes the local particle density.

When $p=3$, choosing sufficiently small $0 < \eta \leq 1 - \frac1r$ for given $r \in (1,\infty]$ and using the Poincar\'e inequality gives
\[
\intt |u - u_c|^2 \rho_f\,dx \leq \|\rho_f\|_{L^{\frac{1}{1-\eta}}} \|u - u_c\|_{L^{\frac1\eta}}^2 \leq C\|\rho_f\|_{L^r}\|Du\|_{L^p}^2.
\]

For $p>3$, we use $W^{1,p}(\T^3) \subseteq L^\infty(\T^3)$ to deduce
\[
\intt |u - u_c|^2 \rho_f\,dx \leq \|\rho_f\|_{L^1} \|u - u_c\|_{L^\infty}^2 \leq C\|\rho_f\|_{L^r}\|Du\|_{L^p}^2.
\]

Thus, we obtain
\[
\calE \leq \inttr |u-v|^2 f\,dxdv + C\|Du\|_{L^p}^2,
\]
which implies
\[
\calE \leq C(\calD + \calD^{\frac 2p}).
\]
Hence, when $\calD(t) \le 1$, 
\[
\frac{d}{dt}\calE(t) + C_0 \calE^{\frac p2}(t) \leq 0,
\]
and when $\calD(t) >1$, 
\[
\frac{d}{dt}\calE(t) + C_0 \calE(t) \leq 0.
\]
Solving this differential inequality concludes the desired decay estimate for $\calE$. 

\subsubsection{The case $2 \geq p > \frac65$} When $p \le 2$, the previous argument does not apply directly due to the loss of monotonicity in the exponent. Instead, we claim
\[
\calE(t) \ls \calD(t), \quad \forall\, t \geq 0.
\]

For $p = 2$, this follows immediately from the prior estimates.  For $p \in (\frac{6}{5}, 2)$, we first choose any $r \in (\frac6{5p-6}, \infty]$. Here, we also choose $\delta \in (0,1)$ small enough so that $(5-\delta)p - 6 > 0$ and 
\[
\ell := \frac{3(2-\delta)}{(5-\delta)p - 6} < r \in  \left(\frac{6}{5p-6}, \infty\right].
\]
Note that 
\[
\ell <\frac{6}{5p-6} 
\]
due to $p \in (\frac{6}{5}, 2)$. For this $\ell$, we apply the Gagliardo--Nirenberg interpolation inequality:
\[
\|u-u_c\|_{L^{\frac{2\ell}{\ell-1}}} \le C\|Du\|_{L^p}^\alpha \|u-u_c\|_{L^q}^{1-\alpha}, \quad \alpha := \frac{\frac1q + \frac{1}{2\ell} - \frac12}{\frac13 + \frac1q - \frac1p},
\]
where  
\[
q := \frac{6-3p}{p -\frac 3\ell} = 2 - \delta \in (1,2) \quad \mbox{and} \quad 2\alpha = p. 
\]
Thus, we obtain
\[
\intt \rho_f |u-u_c|^2 \,dx \le \|\rho_f\|_{L^\ell} \|u-u_c\|_{L^{\frac{2\ell}{\ell-1}}}^2 \le C\|\rho_f\|_{L^r}\|Du\|_{L^p}^p \|u-u_c\|_{L^q}^{2-p} \leq C \mathcal{E}(t)^{1 - \frac p2}\|Du\|_{L^p}^p \leq C\|Du\|_{L^p}^p
\]
for some $C>0$ independent of $t$. 

In the borderline case $r = \frac6{5p-6}$, we get $q = 2$ in the above estimate, and thus the same conclusion holds without modification.

Finally, when $p =\frac65$, we have $r = \infty$, and a similar argument yields
\[
\intt \rho_f |u-u_c|^2 \,dx \leq \|\rho_f\|_{L^\infty}\|u-u_c\|_{L^2}^2 \leq C\|Du\|_{L^{\frac65}}^{\frac65}\|u-u_c\|_{L^2}^{\frac54} \leq C\mathcal{E}(t)^{\frac58}\|Du\|_{L^{\frac65}}^{\frac65} \leq C\|Du\|_{L^{\frac65}}^{\frac65},
\]
with $C>0$ independent of $t$.

This together with \eqref{ineq_E} yields
\[
\calE(t) \leq \inttr |u-v|^2 f\,dxdv + C\|Du\|_{L^p}^p \leq C\calD(t).
\]
From which, the exponential decay rate of convergence of $\calE(t)$ is clear.  This completes the proof.

%
%
%
%

\section*{Acknowledgments}
The work of YPC was supported by NRF grant no. 2022R1A2C1002820 and RS-2024-00406821. The work of JJ was supported by NRF grant no. 2022R1A2C1002820. The work of AWK was supported by National Science Centre Poland grant no. UMO-2020/38/E/ST1/00469.

%
%
%
%
\appendix

\section{Proof for the pointwise convergence of the gradient of fluid velocities}\label{app.A}
\setcounter{equation}{0}
In this appendix, we prove the pointwise convergence of $Dw^{n,\e}$ to $Dw^\e$, following the argument of \cite{FMS97}. A similar strategy applies to the passage $\e \to 0$ yielding $Dw^\e$ to $Dw$. For simplicity of notation, we omit the dependence on $\e$ in this section. Our goal is to show the existence of a subsequence of $\{w^n\}_n$ such that for each $\kappa\in(0,1)$ and $\te\in (0,1)$, there exists $N \in \N$ satisfying
\[
\int_0^T \lt[\intt (\tau (Dw^\ell) - \tau(Dw^k))\cdot (D(w^\ell - w^k))\,dx \rt]^\kappa dt  < \te,
\]
whenever $w^\ell$ and $w^k$ are in the aforementioned subsequence  with $\ell$, $k \ge N$. Once this estimate holds, one combines this with the strict monotonicity \eqref{st.mt} to yield the desired pointwise convergence.

 First, we recall from the energy estimates that there exists a constant $K>0$ independent of $n$ and $\e$ such that
 
\[
\begin{aligned}
&\|w^{n,\e}\|_{L^\infty(0,T;L^2(\T^3))} + \|\nabla_x w^{n,\e}\|_{L^p((0,T)\times\T^3)} + \|w^{n,\e}\|_{L^{\frac{5p}{3}}((0,T)\times\T^3)} \le K,\\
&\|(\theta_\e \star w^{n,\e})\cdot\nabla_x w^{n,\e}\|_{L^s((0,T)\times\T^3)} + \|\pa_t w^{n,\e}\|_{\mathcal{Y}^*} \le K.
\end{aligned}
\]
Now, for $\ell, k \in \N$ we define $g^{\ell,k}$ as 
\[
g^{\ell,k} := |\nabla_x w^\ell|^p + |\nabla_x w^k|^p + (1+ |Dw^\ell|^{p-1} + |Dw^k|^{p-1})(|Dw^\ell| + |Dw^k|).
\]
Then this implies $g^{\ell,k} \in L^1((0,T)\times\T^3)$ and the arguments in \cite[Section 3]{FMS97} that there exists a subsequence  of $\{w^n\}_n$ such that for each $\te\in(0,1)$ and $w^\ell$ and $w^k$ in the subsequence, there exists $L \le \frac{\te}{K}$ which is independent of $\ell$ and $k$ such that
\[
\|g^{\ell,k}\|_{L^1(E^{\ell, k})} \le \te,
\]
where $E^{\ell,k} := \{ (t,x) \in [0,T)\times\T^3 : L^2 \le |w^\ell(t,x) - w^k(t,x)| <L \}$. For each $w^\ell$ and $w^k$, we construct a new sequence $\Psi^{\ell,k}$ as
\[
\Psi^{\ell, k} := (w^\ell - w^k) \lt(1- \min\lt\{1, \frac1L |w^\ell - w^k| \rt\} \rt).
\]
One can check $\Psi^{\ell,k}$ satisfies the following properties:

\begin{enumerate}
\item
$\Psi^{\ell,k} \equiv 0$ on $Q_L^{\ell,k} := \{ (t,x) \in [0,T)\times\T^3  :  |w^\ell - w^k| \ge L\}$ and hence $|\Psi^{\ell,k}|\le L$ otherwise.
\item
$\Psi^{\ell,k} \in L^p(0,T; V^{1,p})$ uniformly in $\ell$ and $k$, $\Psi^{\ell, k} \in L^\infty((0,T)\times\T^3)$,   and $\|\Psi^{\ell,k}\|_{L^r((0,T)\times\T^3)} \to 0$ as $\ell$, $k\to \infty$ for each $r \in [1,\infty)$.
\item
$\|\nabla_x \cdot \Psi^{\ell, k}\|_{L^p((0,T)\times\T^3)} \le 2\te.$
\end{enumerate}
Using this sequence, we consider the solution to the following elliptic problem for almost all $t \in [0,T)$:
\[
-\Delta h^{\ell, k} = -\nabla_x \cdot \Psi^{\ell,k}, \quad \intt h^{\ell, k}\,dx =0.
\]
Then one easily obtains the estimates
\[
\|h^{\ell, k}(t)\|_{W^{2,p}(\T^3)} \le C\|\nabla_x \cdot \Psi^{\ell, k}\|_{L^p(\T^3)}, \quad \|h^{\ell,k}(t)\|_{W^{1,s'}(\T^3)} \le C\|\Psi^{\ell,k}\|_{L^{s'}(\T^3)}
\]
with $C$ independent of $t$, $\ell$, and $k$. Finally, we define
\[
\Phi^{\ell,k} := \Psi^{\ell,k} - \nabla h^{\ell,k}.
\]
This $\Phi^{\ell,k}$ satisfies
\[
\nabla_x \cdot \Phi^{\ell,k} = 0, \quad \Phi^{\ell,k} \in L^p(0,T;V^{1,p}) \cap L^{s'}((0,T)\times\T^3),
\]
and for sufficiently large $\ell$ and $k$,  
\bq\label{Psi_ep}
\|\Phi^{\ell,k}\|_{L^{s'}((0,T)\times\T^3)} \le \te.
\eq
Now, we use $\Phi^{\ell,k}$ as a test function for the weak formulation of $w^\ell$ and $w^k$ so that 
\begin{align*}
&\int_0^T\intt (\pa_t w^\ell - \pa_t w^k)\cdot \Psi^{\ell,k}\,dxdt + \int_0^T \intt (\tau(Dw^\ell) - \tau(Dw^k)): D(w^\ell - w^k) \cdot \chi\,dxdt\\
 &\quad= \int_0^T \intt (\pa_t w^\ell - \pa_t w^k)\cdot \nabla_x  h^{\ell,k}\,dxdt\\
 &\qquad + \int_0^T \intt (\tau(Dw^\ell) - \tau(Dw^k)): D(\nabla_x h^{\ell,k})\,dxdt\\
 &\qquad + \int_0^T \intt (\tau(Dw^\ell) - \tau(Dw^k)): D(w^\ell - w^k) \frac{|w^\ell - w^k|}{L}\chi\,dxdt\\
 &\qquad +  \int_0^T \intt ((u_c^\ell \cdot \nabla_x) w^\ell  - (u_c^k \cdot \nabla_x ) w^k)  \cdot \Phi^{\ell,k}\,dxdt\\
 &\qquad + \int_0^T \intt ( ((\theta_\e \star w^\ell)\cdot \nabla_x) w^\ell - ((\theta_\e \star w^k)\cdot \nabla_x) w^k)\cdot \Phi^{\ell,k}\,dxdt\\
 &\qquad + \int_0^T \intt [\rho_{f^\ell} (\theta_\e \star w^\ell + u_c^\ell) - m_{f^\ell} - \rho_{f^k} ( \theta_\e \star w^k + u_c^k) + m_{f^k})]\cdot \Phi^{\ell,k}\,dxdt\\
 &\quad=: \sum_{i=1}^6 \mathcal{I}_i,
\end{align*}
where $\chi$ is the characteristic function on $([0,T)\times\T^3) \setminus Q_L^{\ell,k}$.
 
Terms $\mathcal{I}_i$ with $i=1,2,3,5$ can be handled in a similar way to \cite{FMS97}. 
To deal with the above, let us notice that the first term is equal to $H(w^\ell -  w^k)(T) - H(w^\ell -  w^k)(0)$, where $H$ is a nonnegative primitive function to $\Psi^{\ell,k}$. Observe that $H(w^\ell - \pa_t w^k)(0)=0$ and therefore the first term is nonnegative.

By integration by parts $\mathcal{I}_1= 0.$  For $\mathcal{I}_3$, uniform $L^1$ bound on $g^{\ell,k}$ and the definition of $L$ yield 
\[\begin{aligned}
|\mathcal{I}_3| &\leq \lt(\iint_{E^{\ell,k}}+ \iint_{\{|w^\ell - w^k|\le L^2\}}\rt) (\tau(Dw^\ell) - \tau(Dw^k)): D(w^\ell - w^k) \frac{|w^\ell - w^k|}{L}\,dxdt\\
&\le C\|g^{\ell,k}\|_{L^1(E^{\ell,k})} + CL\|g^{\ell,k}\|_{L^1((0,T)\times\T^3)}\le C\tilde \e.
\end{aligned}\]

For $\mathcal{I}_2$, we estimate
\[
\mathcal{I}_2 \le C\|g^{\ell,k}\|_{L^1((0,T)\times\T^3)}^{\frac{1}{p'}}\|h^{\ell,k}\|_{L^p(0,T;W^{2,p}(\T^3))}\le C\|\nabla\cdot\Psi^{\ell,k}\|_{L^p((0,T)\times\T^3)}\le C\tilde \e.
\]

For $\mathcal{I}_4$ and $\mathcal{I}_5$, we may follow similarly. Since $((\theta_\e \star w^\ell)\cdot \nabla_x) w^\ell$ is bounded in $L^s((0,T)\times\T^3)$ and by \eqref{Psi_ep}, we find that the smallness of those terms:
\[
\mathcal{I}_4 \le \|(u_c^\ell \cdot \nabla_x) w^\ell - ( u_c^k \cdot \nabla_x ) w^k\|_{L^{s}((0,T)\times\T^3)}\|\Phi^{\ell,k}\|_{L^{s'}((0,T)\times\T^3)} \le c \te,
\]
\[
\mathcal{I}_5 \le \|((\theta_\e \star w^\ell) \cdot \nabla_x) w^\ell - ( (\theta_\e \star w^k) \cdot \nabla_x ) w^k\|_{L^{s}((0,T)\times\T^3)}\|\Phi^{\ell,k}\|_{L^{s'}((0,T)\times\T^3)}
\le c \te.
\]

For $\mathcal{I}_6$, if $\frac85 < p <2$, we use the drag term $\rho_{f^\ell} (\theta_\e \star w^\ell + u_c^\ell) - m_{f^\ell}$ belongs to $L^s((0,T)\times\T^3)$ uniformly in $\ell$ and $\e$, and the smallness of $\Phi^{\ell,k}$ in $L^{s'}((0,T)\times\T^3)$ to get $\mathcal{I}_6 \le C\te$ for sufficiently large $\ell$ and $k$. When $p>2$, we can get
\[\begin{aligned}
\mathcal{I}_6 &\le  \|[\rho_{f^\ell} (\theta_\e \star w^\ell + u_c^\ell) - m_{f^\ell}] - [\rho_{f^k} ( \theta_\e \star w^k + u_c^k) - m_{f^k}]\|_{L^2(0,T; L^{\frac 65}(\T^3))} \|\Phi^{\ell,k}\|_{L^2(0,T; L^6(\T^3))}\\
&\le C\Big(\|\rho_{f^\ell} - \rho_{f^k}\|_{L^q(0,T;L^{\frac 32}(\T^3))}(\|w^\ell\|_{L^{\frac{2q}{q-2}}(0,T;L^6(\T^3))} + \|u_c^\ell\|_{L^\infty((0,T))})\\
&\qquad  + \|\rho_{f^k}\|_{L^\infty(0,T;L^{\frac32}(\T^3))}(\|w^\ell-w^k\|_{L^2(0,T;L^6(\T^3))} + \|u_c^\ell - u_c^k\|_{L^\infty((0,T))})\\
&\qquad + \|m_{f^\ell} - m_{f^k}\|_{L^2(0,T;L^{\frac65}(\T^3))}\Big)\|\Phi^{\ell,k}\|_{L^p(0,T;V^{1,p})},
\end{aligned}\]
where $q>\frac{1}{\frac12 - \frac{2}{5p-6}}$ is sufficiently large so that
\[
\|w^\ell\|_{L^{\frac{2q}{q-2}}(0,T;L^6(\T^3))} \le  C\|\nabla_x w^\ell\|_{L^p((0,T)\times\T^3)}^{\frac{2}{5p-6}} \|w^\ell\|_{L^M(0,T;L^2(\T^3))}^{\frac{3p-6}{5p-6}}
\]
with $M:= \frac{6q(p-2)}{2(q-1)(5p-6)-4q}$. Then we can use the strong convergences of $\rho_{f^\ell}$, $m_{f^\ell}$ and $w^\ell$ to have again $\mathcal{I}_6 \le \te$ for sufficiently large $\ell$ and $k$. This completes the proof for the desired pointwise convergence.

%
%
%
%

\end{document}